\documentclass[final]{siamltex}
\usepackage{amssymb}
\usepackage{bm}
\usepackage{graphicx}
\usepackage{graphics}
\usepackage{caption2}
\usepackage{psfrag}
\usepackage{enumitem}
\usepackage{listings}
\usepackage{cite}
\usepackage{float}
\usepackage{arydshln}
\usepackage{amsmath}
\usepackage{amscd}
\usepackage{amsfonts}
\usepackage{float,verbatim}
\usepackage{latexsym}
\usepackage{color}
\usepackage{multirow}
\usepackage{booktabs}
\usepackage{lscape}
\usepackage{arydshln}
\usepackage{epstopdf}
\usepackage{cases}
\usepackage{lineno}
\newtheorem{remark}{Remark}[section]

\newtheorem{thm}{Theorem}[section]

\newtheorem{lem}[thm]{Lemma}

\newcommand{\p}{\partial}

\newcommand{\ds}{\displaystyle}

\title{Multiscale solution decomposition of nonlocal-in-time problems with application in numerical computation}

\author{ Mengmeng Liu\thanks{School of Mathematics, Shandong University, Jinan 250100, P. R. China. (Email: liumengmeng423@163.com)}
	\and Jie Ma\thanks{School of Statistics and Mathematics, Shandong University of Finance and Economics, Jinan 250014, P. R. China. (Email: jmamath@sdufe.edu.cn)}
	\and Wenlin Qiu\thanks{School of Mathematics, Shandong University, Jinan 250100, P. R. China. (Email: wlqiu@sdu.edu.cn)}
	\and Xiangcheng Zheng\thanks{School of Mathematics, State Key Laboratory of Cryptography and Digital Economy Security, Shandong University, Jinan, 250100, P. R. China. (Email: xzheng@sdu.edu.cn)}
}

\begin{document}

	\maketitle

	\begin{abstract}
		This work develops a multiscale solution decomposition (MSD) method for nonlocal-in-time problems to separate a series of known terms with multiscale singularity from the original singular solution such that the remaining unknown part becomes smoother. We demonstrate that the MSD  provides a scenario where the smoothness assumption for solutions of weakly singular nonlocal-in-time problems, a commonly encountered assumption in numerous literature of numerical methods that is in general not true for original solutions,  becomes appropriate such that abundant numerical analysis results therein become applicable. From computational aspect, instead of handling solution singularity, the MSD significantly reduces the numerical difficulties by separating and thus circumventing the solution singularity. We consider typical problems, including the fractional relaxation equation, Volterra integral equation, subdiffusion, integrodifferential equation and diffusion-wave equation, to demonstrate the universality of MSD and its effectiveness in improving the numerical accuracy or stability in comparison with classical methods.
	\end{abstract}
	
	
	
	\pagestyle{myheadings}
	\thispagestyle{plain}
	\markboth{M. Liu, J. Ma, W. Qiu, X. Zheng}{Multiscale decomposition of nonlocal problems}

	\section{Introduction}
	
	Nonlocal-in-time problems provide adequate descriptions for memory or hereditary properties in complex processes such that they attract increasing attentions with applications in various fields \cite{Bar,CheChe,Chen,Chung,Kor,Pruss,Renardy1,Tur}. A typical feature of nonlocal-in-time problems with weakly singular kernels is that their solutions usually exhibit initial singularities. For this reason, a general consensus is that it is not appropriate to assume smoothness conditions of the solutions for such problems in numerical analysis, and various numerical treatments, e.g. the locally-refined mesh, have been adopted to accommodate the solution singularity.
	
In this work, we consider the nonlocal-in-time problems with weakly singular kernels and smooth data and handle their singularity in a different way, that is, we perform a multiscale solution decomposition (MSD) for the solutions, which separates a series of known terms with different degrees of singularity from the solution such that the remaining unknown part becomes smoother than the original solution. In particular, the unknown part satisfies the same kind of equation as the original solution with a different forcing term, which is the key to improve the solution regularity. Based on the MSD, the following comments could be addressed:
\begin{itemize}
\item \textbf{When smoothness assumption becomes appropriate?} As the unknown part in MSD and the original solution satisfy the same kind of equation with an improved solution regularity,  the MSD indeed provides a scenario where the smoothness assumptions for solutions of nonlocal-in-time problems with weakly singular kernels becomes valid. Under this scenario,
 abundant numerical analysis results for nonlocal-in-time problems in the literature, which are based on the assumption that the solutions are smooth, are applicable (see e.g. the statements below Theorem \ref{thm2.1} for details).

 \item \textbf{Advantages for numerical computation:} Since the unknown part in MSD is smoother than the original solution, one could obtain a higher numerical accuracy on uniform mesh or the same accuracy under a less-singular (i.e. less locally-refined) mesh with an improved numerical stability, compared with the classical numerical method. Consequently, instead of handling solution singularity, the MSD significantly reduces the numerical difficulties by separating and thus circumventing the solution singularity (see e.g. the statements below Theorem \ref{thm3.1} for details).
\end{itemize}

To demonstrate the effectiveness of the MSD by concrete examples, we perform the following investigations:
\begin{itemize}
\item[$\blacktriangle$] We present the basic idea of MSD by a simple fractional relaxation equation, where a rigorous proof is given to show how the solution regularity is improved. Then the nonuniform L1 scheme for fractional relaxation equation of the unknown part is analyzed with an improved accuracy $\min\big\{2-\alpha, (n+1)r\alpha \big\}$ for some $n\geq 0$, where $n=0$ corresponds to the classical estimate of the nonuniform L1 scheme. Numerical experiments indicate that the MSD could improve numerical accuracy or stability in comparison with the classical method.

\item[$\blacktriangle$] We extend MSD method to  typical nonlocal-in-time problems such as Volterra integral equation and partial differential equation (PDE) models including the subdiffusion, integrodifferential equation and diffusion-wave equation:
\begin{itemize}
	\item[$\bullet$] We show that the MSD could effectively improve the solution singularity of Volterra integral equations such that the corresponding collocation scheme could achieve high-order accuracy under the uniform mesh. Numerical experiments are consistent with the analysis.
	
 \item[$\bullet$] The MSD of subdiffusion follows the same spirit of that of fractional relaxation equation due to their similar structures, and numerical experiments again demonstrate the advantages of the MSD in comparison with the classical nonuniform L1 method.

 \item[$\bullet$] We give a detailed mathematical and numerical analysis for the MSD of integrodifferential equation, where we find that the solution regularity and numerical accuracy will be significantly improved by separating only one singular term from the solutions. This means that the MSD method is particularly suitable for integrodifferential problems. In particular,  the MSD improves the numerical accuracy of the trapezoidal convolution quadrature scheme for integrodifferential equation from $O(\tau^{1+\alpha})$ to $O(\tau^2)$ on uniform partition.

 \item[$\bullet$] We further comment on the MSD of the diffusion-wave equation, which could be transformed to an integrodifferential equation such that the previously proposed methods and results for integrodifferential equation could be applied.
 \end{itemize}
\end{itemize}

The rest of the paper is organized as follows: Section \ref{Sec3} presents basic ideas of the MSD by a simple fractional relaxation equation, where the error estimate for the nonuniform L1 scheme of the fractional derivative under improved regularity is proved in Section \ref{sec4}. In Section \ref{Secq}, the MSD method is extended to typical nonlocal-in-time models such as Volterra integral equation, subdiffusion, integrodifferential equation and diffusion-wave equation. We finally address concluding remarks in the last section.

\textbf{Notations:} Let $\Omega$ be an open bounded domain in $\mathbb R^d$ ($1\leq d\leq 3$) with smooth boundary $\p\Omega$, and let $L^p(\Omega)$ with $1\leq p\leq \infty$ be the Banach space of $p$th power Lebesgue integrable functions on $\Omega$. For $0\leq m\in \mathbb{N}$, let $W^{m,p}(\Omega)$ be the Sobolev space of $L^p$ functions with $m$th weakly derivatives in $L^p(\Omega)$. All spaces are equipped with standard norms \cite{AdaFou}. In particular, we set $ H^m(\Omega) := W^{m,2}(\Omega) $ and $ H_0^m(\Omega) $ be the closure of $C^\infty_0(\Omega)$ in $H^m(\Omega) $. For a Banach space $\mathcal{X}$ and some $ T > 0 $, let $ W^{m,p}(0,T;\mathcal{X}) $ be the space of functions in $ W^{m,p}(0,T) $ with respect to $\|\cdot\|_{\mathcal{X}}$. We also denote $C^m[0,T]$ as the space of $m$th continuously differentiable functions on $[0,T]$. For simplicity, we denote $\|\cdot\|:=\|\cdot\|_{L^2(\Omega)}$ and omit $\Omega$ in notations of spatial norms.

Let $\{\lambda_i\}_{i=1}^{\infty}$ and $\{\phi_{i}\}_{i=1}^{\infty}$ denote the eigenvalues and orthonormal eigenfunctions of $-\Delta:H^2\cap H^1_0\rightarrow L^2$, and define
	\begin{align}\label{zz1}
	\dot{H}^r:=\Big\{v\in L^2:|v|_{\dot{H}^r}^2:=\left((-\Delta)^rv,v\right)=\sum_{i=1}^\infty\lambda_i^r(v,\phi_i)^2<\infty\Big\},
	\end{align}
	equipped with the norm $\| g\| _{\dot{H} ^{r } }: = \big ( \| g\| _{L^{2} }^{2}+ | g| _{\dot{H} ^{r} }^{2}\big) ^{1/ 2}$. It is known that $\dot{H} ^{0} = L^{2}$, $\dot{H} ^2=H^2 \cap H_{0}^{1}$ and for $2j-3/2<r<2j+1/2$ for some $1\leq j\in\mathbb N$, $v\in \dot H^r$ implies $v\in H^r$ and $\Delta^i v=0$ for $0\leq i\leq j-1$ \cite[Appendix 2.4]{Jinbook}.
		In this work, $Q$ denotes a generic positive constant that may assume different values at different occurrences.



	\section{Basic idea of MSD}\label{Sec3}
	We demonstrate the idea of MSD by a simple fractional relaxation equation
	\begin{align}\label{mh0}
		\p_t^\alpha u(t)+\lambda u(t)=f(t),~~t\in (0,T];~~u(0)=0.
	\end{align}
		Here $0<\alpha<1$, $T>0$, $\lambda\in\mathbb R$ and $f$ is a bounded forcing term on $[0,T]$. The fractional derivative $\p_t^\alpha u$ is defined via the convolution symbol $*$
	\begin{align*}
	\p_t^\alpha v:=I_t^{1-\alpha} v',~~I_t^{\nu}v:=\beta_{\nu}*v:=\frac{1}{\Gamma(\nu)}\int_0^t\frac{v(s)ds}{(t-s)^{1-\nu}},~~\beta_\nu(t):=\frac{t^{\nu-1}}{\Gamma(\nu)},~~\nu>0.
\end{align*}
	Note that it suffices to consider the homogeneous initial condition since the inhomogeneous case could be transformed to the homogeneous case by variable substitution.
	\subsection{Theoretical aspect}
	For (\ref{mh0}), we introduce the MSD for some $n\geq 1$
	\begin{equation}\label{mh1}
		u=v+I^\alpha_t \sum_{i=0}^{n-1} (-\lambda I^\alpha_t)^i f,
	\end{equation}
	where the function $v$ remains to be determined. It is clear that if $n=0$, then $u=v$. Some features are immediately observed:
	\begin{itemize}
	\item \textbf{Multiscale singularity}: The summands in (\ref{mh1}) exhibit multiscale initial singularities. For instance, if $f\equiv 1$, then (\ref{mh1}) becomes
	\begin{align*}
		u=v+ \sum_{i=0}^{n-1} (-\lambda )^i\frac{t^{(i+1)\alpha}}{\Gamma((i+1)\alpha+1)},
	\end{align*}
	in which the derivatives of the summands of the same order exhibit singularities of different degrees at $t=0$.
	
	\item \textbf{Numerical feasibility}: As the summands in (\ref{mh1}) are known functions, one could approximate $v$ by certain algorithms and then recover $u$ by (\ref{mh1}), instead of approximating $u$ directly. By this means, the error of $v$ equals to that of $u$ due to the relation (\ref{mh1}) such that, if $v$ has better regularity than $u$ due to the singularity separation by (\ref{mh1}), the numerical approximation will be significantly facilitated, as we will see in the rest of the work.
	\end{itemize}

 We invoke (\ref{mh1}) in (\ref{mh0}) to get
	\begin{align*}
		\p_t^\alpha \Big[v+I^\alpha_t \sum_{i=0}^{n-1} (-\lambda I^\alpha_t)^i f\Big]+\lambda \Big[v+I^\alpha_t \sum_{i=0}^{n-1} (-\lambda I^\alpha_t)^i f\Big]=f.
	\end{align*}
	As $f$ is bounded, $v(0)=u(0)=0$ and we apply the semigroup property of the fractional integral operator to get
	\begin{align*}
		\p_t^\alpha v + \sum_{i=0}^{n-1} (-\lambda I^\alpha_t)^i f+\lambda v-\sum_{i=0}^{n-1} (-\lambda I^\alpha_t)^{i+1} f=f,
	\end{align*}
	which, after canceling the same terms, becomes
$
		\p_t^\alpha v +  f+\lambda v- (-\lambda I^\alpha_t)^{n} f=f,
	$
	i.e.
	\begin{align}\label{mh5}
		\p_t^\alpha v +\lambda v=(-\lambda I^\alpha_t)^{n} f, \quad  t\in (0,T]; \quad v(0)=0.
	\end{align}
	Note that both $u$ and $v$ satisfy the fractional relaxation equation with different forcing terms, which means that the MSD process does not generate new models. However, we will show that $v$ is smoother than $u$.
	
	\begin{theorem}\label{thm2.1}
	The following statements hold:
	\begin{itemize}
 \item If $f(0)$ exists and $|f'|\leq Qt^{-\sigma}$ for some $\sigma<1$, then $|v'|\leq Qt^{(n+1)\alpha-1}$ for $t\in (0,T]$ and $n\geq 1$;

 \item If $f'(0)$ exists and $|f''|\leq Qt^{-\tilde\sigma}$ for some $\tilde\sigma<1$, then $|v''|\leq Qt^{(n+1)\alpha-2}$ for $t\in (0,T]$ and $n\geq 1$.
		\end{itemize}
	\end{theorem}
	
	Before the proof, one could discuss the smoothness assumptions for solutionss to weakly singular nonlocal-in-time problems from the following two aspects:
	\begin{itemize}
	\item It is known that the solution $u$ to (\ref{mh0}) with smooth $f$ has initial singularities, i.e.
$|u'|\leq Qt^{\alpha-1}$ and $|u''|\leq Qt^{\alpha-2}$ \cite{Jinbook}, which coincide with the estimates in Theorem \ref{thm2.1} with $n=0$. Thus, it may be not appropriate to impose smoothness assumptions on $u$ when performing numerical analysis;

\item However, after MSD, the singular part of $u$ is separated such that the remaining part $v$ becomes smoother for $n>0$. In particular, for $n$ satisfying $(n+1)\alpha-2\geq 0$, $v'$ and $v''$ are bounded over $[0,T]$. As both $u$ and $v$ satisfy the fractional relaxation equation with different forcing terms, the smoothness assumption for the fractional relaxation equation with solution $v$ becomes appropriate. 
\end{itemize}

	\textit{Proof of Theorem \ref{thm2.1}}.
		Following \cite{Jinbook}, the solution to (\ref{mh5}) could be expressed as
		\begin{align}
			v&=(t^{\alpha-1}E_{\alpha,\alpha}(-\lambda t^\alpha))*(-\lambda I^\alpha_t)^{n} f =(-\lambda)^n (t^{\alpha-1}E_{\alpha,\alpha}(-\lambda t^\alpha))*\beta_{n\alpha}* f  \nonumber\\
			&=(-\lambda)^n (t^{(n+1)\alpha-1}E_{\alpha,(n+1)\alpha}(-\lambda t^\alpha))* f.\nonumber
		\end{align}
		Differentiate this equation to get
		\begin{align*}
			v'=(-\lambda)^n t^{(n+1)\alpha-1}E_{\alpha,(n+1)\alpha}(-\lambda t^\alpha)f(0)+(-\lambda)^n (t^{(n+1)\alpha-1}E_{\alpha,(n+1)\alpha}(-\lambda t^\alpha))* f',
		\end{align*}
		which, together with $|f'|\leq Qt^{-\sigma}$, leads to
		\begin{align*}
			|v'|&\leq Qt^{(n+1)\alpha-1}+Q t^{(n+1)\alpha-1}*t^{-\sigma}\leq Qt^{(n+1)\alpha-1}.
		\end{align*}
		
		To bound $v''$, we differentiate the expression of $v'$ to get
		\begin{align*}
			v''&=(-\lambda)^n t^{(n+1)\alpha-2}E_{\alpha,(n+1)\alpha-1}(-\lambda t^\alpha)f(0) +(-\lambda)^n t^{(n+1)\alpha-1}E_{\alpha,(n+1)\alpha}(-\lambda t^\alpha) f'(0)\\
			&\quad+(-\lambda)^n (t^{(n+1)\alpha-1}E_{\alpha,(n+1)\alpha}(-\lambda t^\alpha))* f'',
		\end{align*}
		which, together with $|f''|\leq Qt^{-\tilde \sigma}$, implies
		\begin{align*}
			|v''|&\leq Q t^{(n+1)\alpha-2}+Q t^{(n+1)\alpha-1}+Q t^{(n+1)\alpha-1}* t^{-\tilde \sigma}\leq Q t^{(n+1)\alpha-2}.
		\end{align*}
		Thus we complete the proof.

	\subsection{Numerical aspect}
\label{sec32}
		Let $0<M\in\mathbb N$  and we introduce the graded mesh
	\begin{align*}
		t_m = (m\tau)^{r}, \quad \tau = \frac{T^{1/r}}{M}, \quad  0\leq m \leq M,
	\end{align*}
	where $r\geq 1$ refers to the mesh grading parameter.  Let $\tau_m = t_m - t_{m-1}$ with $1\leq m \leq M$. We  follow \cite[Equation (5.1)]{Stynes} to get
	\begin{align}\label{wl01}
		\tau_{k+1} = T\left( \frac{k+1}{M} \right)^r - T\left( \frac{k}{M} \right)^r \leq Q TM^{-r}k^{r-1},~~0\leq k \leq M-1.
	\end{align}
	 Then we discretize $\partial_t^{\alpha}v(t_m)$ by the nonuniform L$1$ formula  for $ 1\leq m \leq M$ \cite{Stynes}
	\begin{align}
		D_{M}^{\alpha} v^m & = \frac{1}{\Gamma(2-\alpha)} \sum_{k=0}^{m-1} \frac{v^{k+1}-v^k}{\tau_{k+1}} [ (t_m-t_k)^{1-\alpha} - (t_m -t_{k+1})^{1-\alpha} ]  =  \sum\limits_{k=1}^{m} a^{(m)}_{m-k} \delta_{\tau}v^k  \nonumber
	\end{align}
	where $v^m:= v(t_m)$, $\delta_{\tau}v^k=v^{k}-v^{k-1}$ for $k\geq 1$ and the nonincreasing sequence $\big\{a^{(m)}_{m-k}\big\}_{k=1}^{m}$ is defined by
	\begin{equation*}
			a^{(m)}_{m-k} = \int_{t_{k-1}}^{t_k} \frac{\beta_{1-\alpha}(t_m-s)}{\tau_k}ds = \frac{\beta_{2-\alpha}(t_m-t_{k-1})-\beta_{2-\alpha}(t_m-t_{k})}{\tau_k}.
	\end{equation*}
	The truncation error could be decomposed as follows
	\begin{align}\label{qwl05}
		(R_1)^m := D_{M}^{\alpha} v^m - \partial_t^{\alpha}v^m = \sum_{k=0}^{m-1} T_{m,k},
	\end{align}
	where
	\begin{align}\label{wl03}
		T_{m,k} = \frac{1}{\Gamma(1-\alpha)} \int_{t_k}^{t_{k+1}} (t_m-s)^{-\alpha} \left( \frac{v^{k+1}-v^k}{\tau_{k+1}} -  v'(s)  \right)ds.
	\end{align}
By \eqref{qwl05}, the model (\ref{mh0}) at $t=t_m$ could be written as
	\begin{align}
		& D_M^\alpha v^m  +\lambda v^m =(-\lambda I^\alpha_t)^{n} f(t_m) + (R_1)^m, ~ 1\leq m \leq M;  \label{wl06}  ~ v^0 = v(0).
	\end{align}
	We omit the truncation error and replace $v^m$ by numerical solution $V^m$ to get the discrete scheme
	\begin{align}
		& D_M^\alpha V^m  +\lambda V^m =(-\lambda I^\alpha_t)^{n} f(t_m),~ 1\leq m \leq M;  \label{wl08} ~ V^0 = v(0).
	\end{align}
By \eqref{mh1}, the numerical solution for $u^m:=u(t_m)$ is given as
		\begin{align}\label{mm1}
			U^m=V^m+I^\alpha_t \sum_{i=0}^{n-1} (-\lambda I^\alpha_t)^i f(t_m).
	\end{align}
	
\begin{remark}\label{rem3}
For the special case $n=0$, the scheme (\ref{wl08})--(\ref{mm1}) degenerates to the nonuniform L1 scheme for the original model (\ref{mh0}), which has been widely studied in the literature.
\end{remark}

	To facilitate numerical analysis, we first estimate the truncation error as follows.
	 \begin{lem}\label{lem2.1} Under the assumptions in Theorem \ref{thm2.1}, we have
		\begin{align*}
			|(R_1)^m| \leq Q (t_m)^{n\alpha}  m^{-\min\{2-\alpha, (n+1)r\alpha \}} , \quad 1\leq m \leq M.
		\end{align*}
	\end{lem}
	The proof of this lemma is given in Section \ref{sec4}.	Then we introduce the  complementary discrete convolution kernel \cite{Liao,Liao1}
	\begin{equation}\label{doc}
	P_{m-k}^{(m)} = \frac{1}{a_0^{(k)}}
		\begin{cases}
			1, & k=m, \\
			\sum\limits_{j=k+1}^{m}\left( a_{j-k-1}^{(j)} - a_{j-k}^{(j)}  \right)P_{m-j}^{(m)}, & 1\leq k \leq m-1.
		\end{cases}
	\end{equation}

	\begin{lemma}\cite{Liao,Liao1} \label{lem2.2}
		The convolution kernel \eqref{doc} satisfies
		\begin{align*}
		0< P_{m-k}^{(m)} \leq \Gamma(2-\alpha) \tau_k^{\alpha},~\sum_{j=k}^{m}P_{m-j}^{(m)}a_{j-k}^{(j)}=1,~1\leq k \leq m;\\[-0.1in]
			\sum_{k=1}^{m}  P_{m-k}^{(m)} \beta_{1+q\alpha - \alpha}(t_k) \leq \beta_{1+q\alpha}(t_m), ~ 1\leq m \leq M,~q=0,1.
		\end{align*}
	\end{lemma}
	By Lemmas \ref{lem2.1} and \ref{lem2.2}, we get the following auxiliary result.
	
	\begin{lemma} \label{lem2.3}
		Under the assumptions in Theorem \ref{thm2.1}, we have
		\begin{align*}
			\sum_{m=1}^{m_0}  P_{m_0-m}^{(m_0)} \left|(R_1)^m\right| \leq Q \tau^{-\min\{2-\alpha,r(n+1)\alpha \}}, \quad 1\leq m \leq  m_0 \leq M.
		\end{align*}
	\end{lemma}
	\begin{proof}
		We utilize Lemma \ref{lem2.1} to get
		\begin{equation*}
			\begin{split}
				\sum_{m=1}^{m_0}  P_{m_0-m}^{(m_0)} \left|(R_1)^m\right|  & \leq Q \sum_{m=1}^{m_0}  P_{m_0-m}^{(m_0)} (t_m)^{n\alpha}  m^{-\min\{2-\alpha, (n+1)r\alpha \}}   \\
				& \leq Q(T)\tau^{-\min\{2-\alpha,(n+1)r\alpha \}} \sum_{m=1}^{m_0}  P_{m_0-m}^{(m_0)} (t_m)^{n\alpha}  t_m^{-\min\{\frac{2-\alpha}{r},(n+1)\alpha \}} \\
				& \leq Q(T)\tau^{-\min\{2-\alpha,(n+1)r\alpha \}} \sum_{m=1}^{m_0}  P_{m_0-m}^{(m_0)} t_m^{-\min\{\frac{2-\alpha-n\alpha r}{r}, \alpha \}}.
			\end{split}
		\end{equation*}
		Since
				\begin{equation*}
			\sum_{m=1}^{m_0}  P_{m_0-m}^{(m_0)} t_m^{-\min\{\frac{2-\alpha-n\alpha r}{r}, \alpha \}} \leq Q  \sum_{m=1}^{m_0}  P_{m_0-m}^{(m_0)} t_m^{-\alpha},
		\end{equation*}
		 we apply Lemma \ref{lem2.2} to get
		\begin{align*}
			\sum_{m=1}^{m_0}  P_{m_0-m}^{(m_0)} t_m^{-\alpha} &= \Gamma(1-\alpha) \sum_{m=1}^{m_0}  P_{m_0-m}^{(m_0)} \beta_{1-\alpha}(t_m)  \leq Q\beta_{1}(t_{m_0}) \leq  Q.
		\end{align*}
		We combine the above estimates to complete the proof.
	\end{proof}

	Now we derive the convergence results.
	
	\begin{theorem}\label{thm3.1}
		Let $u^m$ be the exact solution of \eqref{mh0} and $U^m$
		be its numerical approximation given by (\ref{mm1}) Then we have
		\begin{align}\label{jy6}
			|u^{m}-U^m | \leq Q\tau^{\min\{2-\alpha,r(n+1)\alpha \}}, \quad 1\leq m \leq M.
		\end{align}
	\end{theorem}
	
	Before proving this theorem, it is worth mentioning that if $n=0$, (\ref{jy6}) degenerates to $|u^{m}-U^m | \leq Q\tau^{\min\{2-\alpha,r\alpha \}}$,
		which has been obtained in the literature, see e.g. \cite[Lemma 5.2]{Stynes}. Compared with this classical result, the estimate (\ref{jy6}) with $n>0$ has advantages in two main aspects:
		\begin{itemize}
		\item[\textbf{A}.]  \textbf{Higher accuracy on uniform mesh}: If we consider the uniform mesh (i.e. $r=1$), then the classical result (i.e. (\ref{jy6}) with $n=0$) gives the $\alpha$-order accuracy, while for $n$ satisfying $(n+1)\alpha\geq 2-\alpha$, (\ref{jy6}) gives the $(2-\alpha)$-order accuracy.
		Thus, a higher-order accuracy could be attained under uniform mesh.
		
		\item[\textbf{B}.] \textbf{Same accuracy under less-singular mesh}: To obtain $(2-\alpha)$-order accuracy, the classical result (i.e. (\ref{jy6}) with $n=0$) requires $r\geq (2-\alpha)/\alpha$,  while a relaxed condition $r\geq (2-\alpha)/((n+1)\alpha)$ is needed by (\ref{jy6}) with $n>0$. Thus the $(2-\alpha)$-order accuracy could be obtained under a less-singular (i.e. less locally-refined) mesh, which helps to improve the numerical stability.
		\end{itemize}
	
	\textit{Proof of Theorem \ref{thm3.1}}.
	As $u^{m}-U^m =	v^{m}-V^m $, it suffices to consider $v^{m}-V^m$.
	We subtract \eqref{wl08} from \eqref{wl06} to get the error equation
	\begin{align}
		& D_M^\alpha \rho^m  +\lambda \rho^m =  (R_1)^m\text{ for }1\leq m \leq M  \label{wl18} \text{ with }  \rho^0 = 0,\text{ where }\rho^m := v^m -V^m.
	\end{align}
		We first multiply \eqref{wl18} with $2 \rho^m$ to get
		\begin{align*}
			2\rho^m  D_M^\alpha \rho^m  + 2\lambda (\rho^m)^2 =  2 (R_1)^m \rho^m.
		\end{align*}
		Then we multiply this by $P^{(m_0)}_{m_0-m}$ and sum for $m$ from $m=1$ to $m=m_0$ to obtain
		\begin{align*}
			2\sum_{m=1}^{m_0} P^{(m_0)}_{m_0-m} \rho^m  D_M^\alpha \rho^m + 2 \lambda \sum_{m=1}^{m_0} P^{(m_0)}_{m_0-m} (\rho^m)^2 =
			2 \sum_{m=1}^{m_0} P^{(m_0)}_{m_0-m} (R_1)^m \rho^m.
		\end{align*}
		By \cite[Theorem 2.1]{Liao}, we have $2\rho^m  D_M^\alpha \rho^m \geq \sum_{k=1}^{m}a_{m-k}^{(m)}\delta_{\tau}(|\rho^k|^2)$, which follows Lemma \ref{lem2.2} to get
		\begin{align*}
			2&\sum_{m=1}^{m_0} P^{(m_0)}_{m_0-m} \rho^m  D_M^\alpha \rho^m  \geq \sum_{m=1}^{m_0} P^{(m_0)}_{m_0-m}\sum_{k=1}^{m}a_{m-k}^{(m)}\delta_{\tau}(|\rho^k|^2) \\
& = \sum_{k=1}^{m_0} \Big(\sum_{m=k}^{m_0}P^{(m_0)}_{m_0-m}a_{m-k}^{(m)} \Big) \delta_{\tau}(|\rho^k|^2)  =  \sum_{k=1}^{m_0} \delta_{\tau}(|\rho^k|^2) = |\rho^{m_0}|^2 - |\rho^0|^2 = |\rho^{m_0}|^2.
		\end{align*}
		Combine the above two formulas to obtain
		\begin{align*}
		   |\rho^{m_0}|^2 \leq  2 \sum_{m=1}^{m_0} P^{(m_0)}_{m_0-m} |(R_1)^m| |\rho^m|.
		\end{align*}
		By choosing $\bar{m}$ such that $\ds|\rho^{\bar{m}}|=\max_{1\leq m \leq m_0}|\rho^m|$, we have
		$
			|\rho^{\bar{m}}| \leq  2 \sum_{m=1}^{\bar{m}} P^{(\bar{m})}_{\bar{m}-m} |(R_1)^m|,
		$
		which, together with Lemma \ref{lem2.3}, completes the proof.

	\subsection{Numerical experiment}\label{sec34}
Let $T=1$, $f=1$ and $\lambda=1$. We measure errors and convergence orders by the two mesh principle \cite[Page 107]{Farrell}
	\begin{equation}\label{zz2}
		 {\rm Error}_M = \max_{1\leq m\leq M}|U^{2m}(2M)-U^m(M)|, \quad  {\rm Rate} = \log_{2}\left(\frac{{\rm Error}_M}{{\rm Error}_{2M}} \right).
	\end{equation}
	
	We first consider the uniform mesh ($r=1$) in Table \ref{tab:example1_case1} with $\alpha=0.25$ and $\alpha=0.75$, which indicates that the scheme (\ref{wl08})--(\ref{mm1}) with $n=\lceil \frac{2-\alpha}{\alpha} \rceil-1$ approximately exhibits a $(2-\alpha)$-order convergence under the uniform partition, while the uniform L1 method for the original model (\ref{mh0})  (i.e. the scheme (\ref{wl08})--(\ref{mm1}) with $n=0$, cf. Remark \ref{rem3}) only has an $\alpha$-order convergence. These results justify the advantage \textbf{A} below (\ref{jy6}).
	
	We then intend to recover the $(2-\alpha)$-order convergence by the nonuniform mesh in Table \ref{tab:example1_case2}, where we find that both methods could achieve the $2-\alpha$ order for $\alpha=0.75$ by selecting suitable $r$ and $n$ suggested by Theorem \ref{thm3.1}. However, for a relative small $\alpha$ ($\alpha=0.25$ in this experiment), the nonuniform L1 method for the original model (\ref{mh0})  (i.e. the scheme (\ref{wl08})--(\ref{mm1}) with $n=0$) with $r=(2-\alpha)/\alpha=7$ exhibits an unstable convergence behavior due to the singularity (i.e. the severe local refinement) of the mesh, while the scheme (\ref{wl08})--(\ref{mm1}) with $n=3$ and $r=(2-\alpha)/((n+1)\alpha)=1.75$ recovers the $(2-\alpha)$-order convergence due to the less singularity of the mesh. These results agree with the advantage \textbf{B} below (\ref{jy6}).

	\begin{table}
		\centering \footnotesize
		\caption{Numerical methods for fractional relaxation equation with $r=1$.}
		\renewcommand{\arraystretch}{1.2} 
      \begin{tabular}{ccccccccc}
		\hline
		&	\multicolumn{3}{c}{Scheme \eqref{wl08}--\eqref{mm1} with $n=0$}&&\multicolumn{3}{c}{Scheme \eqref{wl08}--\eqref{mm1} with $n=\lceil \frac{2-\alpha}{\alpha} \rceil-1$ }\\
		\cline{2-4}	\cline{6-8}
		&	$ M $ & ${\rm Error}_M$   & $\rm{Rate}$ 	&&	$ M $ & ${\rm Error}_M$   & $\rm{Rate}$   \\
		\hline
		&	128     & $ 2.06 \times 10^{-2} $ &  *  	    &&	128     & $ 6.33 \times 10^{-6} $ &  *     \\
		&	256     & $1.83 \times 10^{-2} $ & 0.17       &&	256     & $ 2.08 \times 10^{-6} $ & 1.60    \\
		$\alpha=0.25$	&	512    & $ 1.61 \times 10^{-2} $  & 0.18      &&	512     & $ 6.80 \times 10^{-7} $ & 1.62 \\
		&	1024    & $ 1.41 \times 10^{-2} $ & 0.19      &&	1024    & $ 2.20 \times 10^{-7} $ & 1.63 \\
		&	2048    & $ 1.23 \times 10^{-2} $ & 0.20      &&	2048    & $ 7.07 \times 10^{-8} $ & 1.64 \\
		&Theory&&0.25&&&&1.75\\
		\hline
		&	128     & $2.17 \times 10^{-3} $ &  *          &&	128     & $1.89\times 10^{-4} $ &  *   \\
		&	256     & $1.29 \times 10^{-3} $ & 0.75       &&	256     & $8.27\times 10^{-5} $ & 1.20 \\
		$\alpha=0.75$	&	512    & $7.65 \times 10^{-4} $  & 0.75       &&	512    & $3.57\times 10^{-5} $ & 1.21\\
		&	1024    & $4.54 \times 10^{-4} $ & 0.75     &&	1024    & $ 1.53 \times 10^{-5} $ & 1.22\\
		&	2048    & $2.70 \times 10^{-4} $ & 0.75      &&	2048    & $ 6.50 \times 10^{-6} $ & 1.23\\
		&Theory&&0.75&&&&1.25\\
		\hline
	\end{tabular}
		\label{tab:example1_case1}%
	\end{table}

	\begin{table}
		\centering \footnotesize
		\caption{Numerical methods for fractional relaxation equation with $r=\frac{2-\alpha}{(n+1)\alpha}$.}
		\renewcommand{\arraystretch}{1.1} 
		\begin{tabular}{ccccccccc}
			\hline
			&	\multicolumn{3}{c}{Scheme \eqref{wl08}--\eqref{mm1} with $n=0$}&&\multicolumn{3}{c}{Scheme \eqref{wl08}--\eqref{mm1} with $n=3$ }\\
			\cline{2-4}	\cline{6-8}
			&	$ M $ & ${\rm Error}_M$   & $\rm{Rate}$ 	&&	$ M $ & ${\rm Error}_M$   & $\rm{Rate}$   \\
			\hline
			&	128     & $ 9.33 \times 10^{-5} $ &  *  	    &&	128     & $ 1.88 \times 10^{-6} $ &  *     \\
			&	256     & $3.33 \times 10^{-5} $ & 1.49       &&	256     & $ 5.86 \times 10^{-7} $ & 1.68    \\
			$\alpha=0.25$	&	512    & $ 2.53 \times 10^{-5} $  & 0.40      &&	512     & $ 1.81 \times 10^{-7} $ & 1.69 \\
			&	1024    & $ 1.57 \times 10^{-5} $ & 0.69      &&	1024    & $ 5.56 \times 10^{-8} $ & 1.70 \\
			&	2048    & $ 9.29 \times 10^{-6} $ & 0.76      &&	2048    & $ 1.70 \times 10^{-8} $ & 1.71 \\
			&Theory&&1.75&&&&1.75\\
			\hline
			&	128     & $5.34 \times 10^{-4} $ &  *          &&	128     & $1.95\times 10^{-4} $ &  *   \\
			&	256     & $2.35 \times 10^{-4} $ & 1.18       &&	256     & $8.87\times 10^{-5} $ & 1.13 \\
			$\alpha=0.75$	&	512    & $1.03 \times 10^{-4} $  & 1.20       &&	512    & $4.01 \times 10^{-5} $ & 1.14\\
			&	1024    & $4.44 \times 10^{-5} $ & 1.21     &&	1024    & $ 1.80 \times 10^{-5} $ & 1.16\\
			&	2048    & $1.91 \times 10^{-5} $ & 1.22      &&	2048    & $ 8.04 \times 10^{-6} $ & 1.16\\
			&Theory&&1.25&&&&1.25\\
			\hline
		\end{tabular}
		\label{tab:example1_case2}%
	\end{table}

\section{Proof of Lemma \ref{lem2.1}}\label{sec4}
In general, the proof could be performed following the spirit of \cite[Lemma 5.2]{Stynes}, which focuses on the case $n=0$. Nevertheless, subtle estimates are required to accommodate the impact of the regularity estimate in Theorem \ref{thm2.1} with $n>0$. For $1\leq k \leq m-2$, we invoke \cite[Equation (5.6)]{Stynes} to obtain
	\begin{align}\label{wl04}
		|T_{m,k}| \leq Q \tau_{k+1}^{2} \left( \max\limits_{t\in[t_k,t_{k+1}]}|v''(t)|  \right) \int_{t_k}^{t_{k+1}}(t_m-s)^{-\alpha-1}ds.
	\end{align}
	Based on this, we divide the proof into three parts:
		
		\textit{Case I: $1\leq k < m-1$.} For $1\leq k < m-1$, \eqref{wl04} and \eqref{wl01} give
		\begin{equation*}
			\begin{split}
				|T_{m,k}| & \leq Q \tau_{k+1}^{2} t_{k}^{(n+1)\alpha-2} \tau_{k+1} (t_m-t_{k+1})^{-\alpha-1} \\
				& \leq Q \tau_{k+1}^{3}k^{r[(n+1)\alpha-2]} M^{3r-n\alpha r} [m^r - (k+1)^r]^{-\alpha-1},
			\end{split}
		\end{equation*}
		which, together with \eqref{wl01}, leads to
		\begin{equation*}
			\begin{split}
				|T_{m,k}| & \leq  Q M^{-3r} k^{3r-3} k^{r[(n+1)\alpha-2]} M^{3r-n\alpha r} [m^r - (k+1)^r]^{-\alpha-1} \\
				& \leq Q  M^{-n\alpha r}  k^{r[(n+1)\alpha+1] -3}  [m^r - (k+1)^r]^{-\alpha-1}.
			\end{split}
		\end{equation*}
		Then we split the sum to obtain
		\begin{equation*}
			\begin{split}
				\sum_{k=1}^{\lceil m/2 \rceil-1} |T_{m,k}| & \leq Q  M^{-n\alpha r} \sum_{k=1}^{\lceil m/2 \rceil-1} k^{r[(n+1)\alpha+1] -3}  m^{-r(\alpha+1)} \\
				& = Q  M^{-n\alpha r}  m^{-r(\alpha+1)}  \sum_{k=1}^{\lceil m/2 \rceil-1} k^{r[(n+1)\alpha+1] -3},
			\end{split}
		\end{equation*}
	where $\lceil \cdot \rceil$ represents the rounding up symbol.	By \cite[Equation (5.9)]{Stynes}, we have
			\begin{equation}\label{wl12}
			\sum_{k=1}^{\lceil m/2 \rceil-1} |T_{m,k}| \leq Q \times
			\begin{cases}
				M^{-n\alpha r}  m^{-r(\alpha+1)}, &  r[(n+1)\alpha+1] < 2, \\
				M^{-n\alpha r}  m^{n\alpha r -2} \ln m, &  r[(n+1)\alpha+1] = 2,   \\
				M^{-n\alpha r}  m^{n\alpha r -2}, &  r[(n+1)\alpha+1] > 2.
			\end{cases}
		\end{equation}
		For $\lceil m/2 \rceil \leq k \leq m-2$, \eqref{wl04} and \eqref{wl01} give
		\begin{align*}
			|T_{m,k}| & \leq Q (TM^{-r}k^{r-1})^2 t_k^{(n+1)\alpha -2} \int_{t_k}^{t_{k+1}}(t_m-s)^{-\alpha-1}ds  \\
			& \leq Q M^{-(n+1)\alpha r}  m^{r(n+1)\alpha -2} \int_{t_k}^{t_{k+1}}(t_m-s)^{-\alpha-1}ds,
		\end{align*}
		which, together with \eqref{wl01}, leads to
		\begin{equation}\label{wl13}
			\begin{split}
				\sum_{k=\lceil m/2 \rceil}^{m-2} |T_{m,k}| & \leq Q M^{-(n+1)\alpha r}  m^{r(n+1)\alpha -2} \int_{t_{\lceil m/2 \rceil}}^{t_{m-1}}(t_m-s)^{-\alpha-1}ds \\
				& \leq Q  M^{-(n+1)\alpha r}  m^{r(n+1)\alpha -2} (t_m -t_{m-1})^{-\alpha} \\
				& \leq Q  M^{-(n+1)\alpha r}  m^{r(n+1)\alpha -2} (TM^{-r}m^{r-1})^{-\alpha} \\
				& \leq Q  M^{-n \alpha r}  m^{n\alpha r - (2-\alpha)} \leq Q(t_m)^{n\alpha} m^{- (2-\alpha)}.
			\end{split}
		\end{equation}
		By $\ln m \leq Q m^{\alpha}$, we apply \eqref{wl12} and \eqref{wl13} to get
		\begin{equation*}
			\begin{split}
				\sum_{k=1}^{m-2} |T_{m,k}| &  \leq Q (t_m)^{n\alpha} m^{- (2-\alpha)}.
			\end{split}
		\end{equation*}

		\textit{Case II: $k=0$.} We consider the estimate of $T_{m,0}$. When $m=1$, Theorem \ref{thm2.1} and \eqref{wl03} give
		\begin{equation}\label{wl15}
			\begin{split}
			|T_{1,0}| &= \Big|\frac{\tau_1^{-\alpha}}{\Gamma(2-\alpha)}( v^1 - v^0 ) - \frac{1}{\Gamma(1-\alpha)} \int_{0}^{t_1} (t_1-s)^{-\alpha}  v'(s)ds \Big|\\
			&\leq \frac{\tau_1^{-\alpha}}{\Gamma(2-\alpha)} \int_{0}^{\tau_1} |v'(s)|ds +Q \int_{0}^{\tau_1} (t_1-s)^{-\alpha} |v'(s)|ds\\
			&\leq Q \tau_1^{-\alpha} \int_{0}^{\tau_1} s^{(n+1)\alpha-1} ds+Q  \int_{0}^{\tau_1} (t_1-s)^{-\alpha} s^{(n+1)\alpha-1} ds \\
			&\leq Q \tau_1^{n\alpha} +Q (t_1)^{n\alpha}  \int_{0}^{\tau_1} (t_1-s)^{-\alpha} s^{\alpha-1} ds  \leq Q M^{-n\alpha r}.
		\end{split}
		\end{equation}
When $m>1$, we use  Theorem \ref{thm2.1} to estimate $T_{m,0}$ as
\begin{align*}
   |T_{m,0}| & \leq \Big| \frac{\tau_1^{-1} (v^1-v^0)}{\Gamma(2-\alpha)} [t_m^{1-\alpha} - (t_m-t_1)^{1-\alpha}] \Big| + \Big|  \frac{1}{\Gamma(1-\alpha)} \int_{0}^{t_1} (t_m-s)^{-\alpha}  v'(s)ds  \Big| \\
   & \leq Q \tau_1^{-1} [t_m^{1-\alpha} - (t_m-t_1)^{1-\alpha}] \int_{0}^{\tau_1} |v'(s)|ds + Q \frac{(t_m-t_1)^{-\alpha}}{\Gamma(1-\alpha)} \int_{0}^{\tau_1} |v'(s)|ds \\
   & \leq  Q \frac{t_m^{1-\alpha} - (t_m-t_1)^{1-\alpha}}{t_1} (t_1)^{(n+1)\alpha} + Q (t_1)^{\alpha(n+1)-\alpha} \left( \frac{t_m-t_1}{t_1} \right)^{-\alpha} \\
   & \leq Q (t_1)^{n\alpha} \left( \frac{t_m-t_1}{t_1} \right)^{-\alpha} + QM^{-n\alpha r} \left( \frac{t_m-t_1}{t_1} \right)^{-\alpha} \\
   & \leq Q   M^{-n\alpha r} \left( \frac{t_m-t_1}{t_1} \right)^{-\alpha} \leq Q M^{-n\alpha r} m^{-r\alpha},
\end{align*}
which, together with \eqref{wl15}, leads to $
				|T_{m,0}|   \leq Q   M^{-n\alpha r} m^{-r\alpha}$ for $ m\geq 1.
	$

		\textit{Case III: $k=m-1$.} To bound $T_{m,m-1}$, we first note that the case $m=1$ has been discussed in \eqref{wl15} and we remain to consider $m>1$. With $\zeta_0, \zeta_1, \zeta_2 \in (t_{m-1},t_m)$, we apply  Theorem \ref{thm2.1} to obtain
\begin{align*}
	|T_{m,m-1}| & \leq  \frac{\big| v'(\zeta_1) - v'(\zeta_2) \big|}{\Gamma(1-\alpha)}  \int_{t_{m-1}}^{t_m} (t_m-s)^{-\alpha}ds \leq Q \tau_m  |v''(\zeta_0)| (t_m -t_{m-1})^{1-\alpha} \nonumber \\
     &   \leq Q \tau_{m}^{2-\alpha} t_{m-1}^{(n+1)\alpha -2}\leq Q M^{-n\alpha r} (m-1)^{n\alpha r - (2-\alpha)}    \\
	 & \leq Q M^{-n\alpha r} m^{n\alpha r - (2-\alpha)} \leq Q m^{- (2-\alpha)} (t_m)^{n\alpha}. \nonumber
\end{align*}
		We summarize the above results to finally obtain
		\begin{align*}
			|(R_1)^m|  & \leq Q \Big[ m^{- (2-\alpha)} (t_m)^{n\alpha} + M^{-n\alpha r} m^{-r\alpha} \Big]  \\
        & \leq Q \Big[ m^{- (2-\alpha)} (t_m)^{n\alpha} + (t_m)^{n\alpha}  m^{-(n+1) r\alpha} \Big]
			 \leq Q (t_m)^{n\alpha}  m^{-\min\{2-\alpha, (n+1)r\alpha \}},
		\end{align*}
		which finishes the proof of the lemma.

	\section{Extension to integral equation and PDEs}\label{Secq}
	We extend the MSD method for other typical nonlocal-in-time models, including the Volterra integral equation and PDEs such as subdiffusion, integrodifferential equation and diffusion-wave equation.
	
	\subsection{Volterra integral equation}
	Consider the following Volterra integral equation for some $0<\alpha<1$ \cite{Bru}
	\begin{align}\label{vie0}
	u(t)=(\mathcal Lu)(t)+f(t):=\int_0^t(t-s)^{-\alpha}K(s,t)u(s)ds+f(t),~~t\in [0,T].
	\end{align}
	Here $K$ and $f$ are sufficiently smooth functions on their domains.
We apply the variable substitution $u=w+f(0)$ to get
\begin{align}\label{vie1}
	w(t)=(\mathcal Lw)(t)+\tilde f(t):=(\mathcal Lw)(t)+(\mathcal L f(0))(t)+f(t)-f(0),~~t\in [0,T],
	\end{align}
	which implies that $w(0)=\tilde f(0)=0$. Now we apply the MSD
	\begin{align}\label{vie11}
	w=v+\sum_{i=0}^{n-1}\mathcal L^i\tilde f
	\end{align}
in (\ref{vie1}) to get
	\begin{align}\label{vie2}
	v(t)=(\mathcal Lv)(t)+(\mathcal L^n\tilde f)(t),~~t\in [0,T].
	\end{align}
	Since $\tilde f$ is a bounded and smooth function, one could conclude that if $n(1-\alpha)\geq m$ for some $0<m\in\mathbb N$, then $\mathcal L^n\tilde f\in C^m[0,T]$ with $(\mathcal L^n\tilde f)^{(i)}(0)=0$ for $0\leq i\leq m-1$. To show this statement, we first consider the special case $K\equiv 1$ such that $\mathcal L^n\tilde f=\Gamma(1-\alpha)^n I_t^{n(1-\alpha)}\tilde f=\Gamma(1-\alpha)^n I_t^m I_t^{n(1-\alpha)-m}\tilde f$. Then $(\mathcal L^n\tilde f)^{(m)}=\Gamma(1-\alpha)^n I_t^{n(1-\alpha)-m}\tilde f$, which is a continuous function for $n(1-\alpha)\geq m$. For $K\not\equiv 1$, more calculus is needed but the smooth coefficient $K$ does not affect the result. Thus one could apply standard analysis methods of Volterra integral equations (cf. \cite{Bru} or \cite[Section 3]{Lia}) to show that $v\in C^m[0,T]$ (with $v^{(i)}(0)=0$ for $0\leq i\leq m-1$).
	
	Now we employ the collocation method \cite{Bru,Lia} to discretize the integral equation \eqref{vie2}. Define a uniform partition $0=t_0<t_1<\cdots<t_M=T$ with time step size $\tau$. We introduce the following piecewise polynomial space
\begin{equation*}
	S^h_{q-1} := \big \{ \chi: \chi|_{(t_m,t_{m+1}]}\in P_{q-1}, \ {\rm for} \ 0\le m \le M-1  \big \},
\end{equation*}
where $0<q\in\mathbb N$ and $P_{q-1}$ denotes the space of all real-valued polynomials of degree at most \(  q-1 \). 
Then the collocation solution \(  V_h \in S^{h}_{q-1}  \) to \eqref{vie2} is defined by 
\begin{equation}\label{vie01}
	V_h(t) = \int_{0}^{t} (t-s)^{-\alpha} K(s,t)V_h(s)ds + (\mathcal{L}^n\tilde{f})(t),\ {\rm for} \ t\in X_h,
\end{equation}
where $X_h$ is the set of collocation points defined by
\begin{equation*}
	X_h:=\{ t_m + c_i\tau : 0<c_1<\cdots< c_q \le1 \ {\rm for}\ 0\leq m\leq M-1  \},
\end{equation*}
for some user-selected collocation parameters \(   \{c_i\}   \).
On each interval \(   (t_m, t_{m+1}]   \), the collocation solution is given by
\begin{equation}\label{numu}
V_{h}(t_{m} + s \tau) = \sum_{i=1}^{q} L_{i}(s) V_{m,i} \ {\rm for } \ s \in (0,1],
\end{equation}
where $V_{m,i} = V_h(t_{m,i}) $ with $t_{m,i}:= t_m + c_i \tau$ for $1\leq i\leq q$ and the polynomials $L_i(s) := \prod_{k=1,k\neq i}^{q} \frac{s-c_k}{c_i - c_k} $ for $1\leq i\leq q $ are the standard Lagrange basis functions for $P_{q-1}$ on $(0,1]$ with
respect to the (distinct) collocation parameters $\{c_i\}$. Then the $\{V_{n,i}\}$ could be determined by invoking (\ref{numu}) in (\ref{vie01}).
    According to \eqref{vie11}, the numerical approximation for \(  u(t_{m,i})  \) is then given as
  	\begin{align}\label{CollMethod1}
  	U_{m,i} = V_{m,i} + u_0 + \sum_{i=0}^{n-1} (\mathcal{L}^i \tilde{f})(t_{m,i}),~~0\leq m\leq M-1,~~1\leq i\leq q.
  \end{align}
  When $n=0$, the scheme (\ref{vie01})--(\ref{CollMethod1}) degenerates to the standard collocation method for the original model (\ref{vie0}). By \cite{Bru,Lia}, the scheme (\ref{vie01})--(\ref{CollMethod1}) with $n$ large enough (such that $v\in C^m[0,T]$) admits an $O(\tau^m)$ accuracy at collocations points, where only $O(\tau^{2(1-\alpha)})$ accuracy could be attained for $n=0$. 
  
  Now we demonstrate the usage of MSD in improving the numerical accuracy by a numerical example.
  Let $T=1$, $K(t,s)=1/\Gamma(1-\alpha)$, $f=1$, and we set collocation parameters as $c_1=2/3$, $c_2=1$, $q=2$. We measure errors
  and convergence orders by the two mesh principle at the mesh points as (\ref{zz2}).
 Numerical  results are presented in Table \ref{tab:mm41}, which agrees well with the predictions for the numerical accuracy and thus indicate the capability of MSD  in improving the numerical accuracy.

  
  \begin{table}
  	\centering \footnotesize
  	\caption{Numerical methods for Volterra integral equation.}
  	\renewcommand{\arraystretch}{1.1} 
  	\begin{tabular}{ccccccccc}
  		\hline
  		&	\multicolumn{3}{c}{Scheme (\ref{vie01})--(\ref{CollMethod1}) with $n=0$ }&&\multicolumn{3}{c}{Scheme (\ref{vie01})--(\ref{CollMethod1}) with $n=\lceil \frac{1}{1-\alpha} \rceil-1$ }\\
  		\cline{2-4}	\cline{6-8}
  		&	$ M $ & ${\rm Error}_M$   & ${\rm Rate}_{\tau}$ 	&&	$ M $ & ${\rm Error}_M$   & ${\rm Rate}_{\tau}$  \\
  		\hline	
  		&512  & $5.68\times 10^{-6}$ & *    && 512  & $9.73\times 10^{-6}$ & *     \\
  		&1024 & $2.12\times 10^{-6}$ & 1.42 && 1024 & $2.44\times 10^{-6}$ & 1.99  \\
  		$\alpha=0.25$&2048 & $7.42\times 10^{-7}$ & 1.52 && 2048 & $6.12\times 10^{-7}$ & 2.00  \\
  		&4096 & $2.48\times 10^{-7}$ & 1.58 && 4096 & $1.53\times 10^{-7}$ & 2.00  \\
  		&8192 & $8.13\times 10^{-8}$ & 1.61 && 8192 & $3.84\times 10^{-8}$ & 2.00 \\
  		&Theory& &1.50& & & &2.00\\
  		\hline
  		&512  & $1.08\times 10^{-3}$ & *    && 512  & $6.61\times 10^{-5}$ & *     \\
  		&1024 & $6.79\times 10^{-4}$ & 0.68 && 1024 & $1.70\times 10^{-5}$ & 1.96  \\
  		$\alpha=0.75$&2048 & $4.34\times 10^{-4}$ & 0.65 && 2048 & $4.35\times 10^{-6}$ & 1.97  \\
  		&4096 & $2.82\times 10^{-4}$ & 0.62 && 4096 & $1.11\times 10^{-6}$ & 1.97  \\
  		&8192 & $1.86\times 10^{-4}$ & 0.60 && 8192 & $2.82\times 10^{-7}$ & 1.98 \\
  		&Theory&&0.50&&&&2.00\\
  		\hline
  	\end{tabular}
  	\label{tab:mm41}%
  \end{table}

\begin{remark}
In \cite{Lia,ZheWan}, a variable-exponent Volterra integral equation with smooth exponent $0\leq \alpha(t)<1$  is considered
\begin{align}\label{vie3}
	u(t)=\int_0^t(t-s)^{-\alpha(t)}K(s,t)u(s)ds+f(t),~~t\in [0,T],
	\end{align}
and it is shown in \cite{Lia,ZheWan} that the condition $\alpha^{(i)}(0)=0$ for $0\leq i\leq m-1$ could eliminate the solution singularity of (\ref{vie3}) and thus ensure $u\in C^m[0,T]$. Under this condition, the scheme (\ref{vie01})--(\ref{CollMethod1}) with $n=0$ for (\ref{vie3}) could reach $O(\tau^m)$ accuracy on the uniform mesh, see \cite[Corollary 4.7]{Lia}. By the MSD as for the constant-exponent model (\ref{vie0}), one could also eliminate the solution singularity by separating singular terms such that the $O(\tau^m)$ accuracy on the uniform mesh could be proved following the same analysis method in \cite{Lia}.
\end{remark}

	\subsection{Subdiffusion}
	Consider the following subdiffusion equation, which describes the anomalous diffusion phenomena in various fields \cite{Deng,Hon,Kor,Zhu}
	\begin{equation}\label{sub1}
		\begin{array}{ll}
			&~~~~\p^\alpha_t u(x,t)-\Delta u(x,t)=f(x,t),~~(x,t)\in \Omega\times (0,T];\\[0.05in]
			&u(x,t)=0,~~(x,t)\in \p\Omega\times [0,T];~~u(x,0)=u_0(x),~~x\in\Omega.
		\end{array}
	\end{equation}
	In general, the subdiffusion share similar structures as the fractional relaxation equation (\ref{mh0}) (due to the positivity of $-\Delta$) such that one could follow the spirit of the derivations in Section \ref{Sec3} to perform mathematical and numerical analysis. Specifically, by variable substitution $u=w+u_0$, we reformulate model (\ref{sub1}) in terms of $w$
 	\begin{equation*}
		\begin{array}{ll}
			&\p^\alpha_t w(x,t)-\Delta w(x,t)= f(x,t)+\Delta u_0(x),~~(x,t)\in \Omega\times (0,T];\\[0.05in]
			&~~~w(x,t)=0,~~(x,t)\in \p\Omega\times [0,T];~~w(x,0)=0,~~x\in\Omega.
		\end{array}
	\end{equation*}
We assume $f+\Delta u_0\in \dot H^{2n}$ for each $t\in [0,T]$ throughout this section such that, based on properties below (\ref{zz1}), $\Delta^j (f+\Delta u_0)=0$ on $\partial\Omega\times [0,T]$ for $0\leq j\leq n-1$ and $\Delta^n (f+\Delta u_0)\in L^2$ for each $t\in [0,T]$.
	In analogous to (\ref{mh1}), we perform the  MSD
	\begin{align}\label{sub2}
		w=v+I^\alpha_t \sum_{i=0}^{n-1} ( I^\alpha_t \Delta)^i (f+\Delta u_0),
	\end{align}
	such that the unknown part $v$ satisfies
	\begin{equation}
		\begin{array}{cc}
			\p^\alpha_t v(x,t)-\Delta v(x,t)= I^{n\alpha}_t\Delta^n (f(x,t)+\Delta u_0(x)),~~(x,t)\in \Omega\times (0,T];\label{sub3}\\
			v(x,t)=0,~(x,t)\in \p\Omega\times [0,T]; ~v(x,0)=0,~x\in\Omega.
		\end{array}
	\end{equation}
	
In general, we one could apply standard methods  such as the Laplace transform method \cite[Section 6.2]{Jinbook} to show that the temporal regularity of $v$ follows that of the solutions to the fractional relaxation equation (\ref{mh0}) as shown in Theorem \ref{thm2.1}. Thus we omit the details due to similarity and focus on the numerical approximation to show the effectiveness of the MSD.
We invoke the nonuniform L1 formula for the fractional derivative in Section \ref{sec32} to discretize \eqref{sub3} at $t=t_m$ as
	\begin{align}
		& D_M^\alpha v^m(x)  - \Delta v^m(x) = I^{n\alpha}_t\Delta^n (f^m(x) +\Delta u_0(x)) + (R_1)^m(x), ~~ 1\leq m \leq M,   \label{mm2} \\
		&\qquad\qquad\qquad~ v^m(x)=0,~x\in \p\Omega; ~v^0(x)=0,~~x\in\Omega,\nonumber
	\end{align}
	where $v^m(x)=v(x,t_m)$, $f^m(x)=f(x,t_m)$, and $(R_1)^m(x)$ is given by \eqref{qwl05}.
	To get a fully discrete scheme,  we denote a quasi-uniform partition of $\Omega$ with the mesh diameter $h$, and let $S_h$ be the space of continuous piecewise linear functions on $\Omega$ with respect to the partition.
 We then integrate \eqref{mm2} multiplied by $\chi \in H_0^1$ on $\Omega$ and drop the truncation errors to get the fully discrete Galerkin scheme: find $V^m_h \in S_h$ with $V^0_h  = 0$ such that the following scheme holds for any $\chi \in S_h$ and $1\leq m \leq M$
	\begin{align}
		&\left(  D_M^\alpha  V_h^m, \chi \right) +  \left( \nabla V_h^{m} , \nabla\chi \right)  =  \left( I^{n\alpha}_t\Delta^n (f^m+\Delta u_0), \chi \right)\label{mm15}.
	\end{align}
Based on \eqref{sub2}, the numerical approximation for $u^m$, which is not a function in $S_h$, could be given as
		\begin{align}\label{mm8}
			U_h^m = V_h^m + u_0 + I^\alpha_t \sum_{i=0}^{n-1} ( I^\alpha_t\Delta)^i (f^m+\Delta u_0).
	\end{align}
In particular, for $n=0$ and $u_0=0$, the scheme (\ref{mm15})--(\ref{mm8}) degenerates to the direct L1 scheme for the original subdiffusion model (\ref{sub1}).

\begin{remark}
 Analogously to the numerical analysis in Section \ref{sec32},  one could derive the error bound for the  scheme \eqref{mm15}--\eqref{mm8} with $O(\tau^{\min\{2-\alpha,r(n+1)\alpha \}} + h^2)$ accuracy, as we will show in the following numerical example.
\end{remark}
	

Now we compare the direct L1 scheme for original model (\ref{sub1}) and the MSD-based scheme \eqref{mm15}--\eqref{mm8} for (\ref{sub1}) by a numerical example. Let $T=1$, $\Omega=(0,2\pi)$, $f=\sin(x)$ and $u_0(x)=\sin(\frac{1}{2}x)$. The uniform partition is adopted in space with $h=2\pi/J$ for some $0<J\in\mathbb N$. We define the discrete $L^\infty(0,T;L^2)$ error as 
\begin{equation}\label{Rate}
	{\rm Error}_M(\tau, h) = \max\limits_{0 \leq m \leq M} \sqrt{h\sum\limits_{j=1}^{J-1} (U_j^{2m}(2M) - U_j^{m}(M))^2}, 
			\end{equation}
and the convergence rate Rate$_{\tau}$ is measured by the two mesh principle as (\ref{zz2}). We fix $J=128$ to focus the attention on the temporal accuracy.

Numerical experiments are performed in the same manner as those in Section \ref{sec34}, and the results presented in Tables \ref{tab:mm1}--\ref{tab:mm2} again indicate the advantages \textbf{A}--\textbf{B} of the MSD stated below (\ref{jy6}).
	
	\begin{table}
		\centering \footnotesize
		\caption{Numerical methods for subdiffusion with $r=1$.}
		\renewcommand{\arraystretch}{1.1} 
		\begin{tabular}{ccccccccc}
			\hline
			&	\multicolumn{3}{c}{Direct L1 scheme for original model}&&\multicolumn{3}{c}{Scheme \eqref{mm15}--\eqref{mm8} with $n=\lceil \frac{2-\alpha}{\alpha} \rceil-1$ }\\
			\cline{2-4}	\cline{6-8}
			&	$ M $ & ${\rm Error}_M(\tau, h)$   & ${\rm Rate}_{\tau}$ 	&&	$ M $ & ${\rm Error}_M(\tau, h)$   & ${\rm Rate}_{\tau}$  \\
			\hline	
			&	64    & $ 4.31 \times 10^{-2} $  & *     &&	64    & $3.38\times 10^{-5} $  &* \\
			&	128    & $ 3.84 \times 10^{-2} $ & 0.17      &&	128    & $1.12 \times 10^{-5} $ & 1.59\\
			$\alpha=0.25$			&	256    & $ 3.39 \times 10^{-2} $ & 0.18      &&	256    & $3.69 \times 10^{-6} $ & 1.60\\
			&	512    & $ 2.98 \times 10^{-2} $ & 0.19      &&	512    & $1.20 \times 10^{-6} $ & 1.62\\
			&	1024     & $ 2.60 \times 10^{-2} $ &  0.20  	    &&	1024     & $3.90 \times 10^{-7} $ &  1.63  \\
			&Theory&&0.25&&&&1.75\\
			\hline
			&	512     & $ 1.40 \times 10^{-3} $ & *      &&	512    & $6.35 \times 10^{-5} $ & *\\
			&	1024    & $ 8.29 \times 10^{-4} $ & 0.75   &&	1024   & $2.72 \times 10^{-5} $ &  1.22\\
			$\alpha=0.75$ &	2048     & $4.93 \times 10^{-4} $ &  0.75  &&	2048   & $1.16 \times 10^{-5} $ &  1.23  \\
			&	4096     & $2.93 \times 10^{-4} $ &  0.75  &&	4096   & $4.91 \times 10^{-6} $ &  1.24\\
			&	8192     & $ 1.74 \times 10^{-4} $ &  0.75 &&	8192   & $2.08 \times 10^{-6} $ &  1.24\\
			&Theory&&0.75&&&&1.25\\
			\hline
		\end{tabular}
		\label{tab:mm1}%
	\end{table}

	\begin{table}
		\centering \footnotesize
		\caption{Numerical methods for subdiffusion with $r=\frac{2-\alpha}{(n+1)\alpha}$.}
		\renewcommand{\arraystretch}{1.1} 
		\begin{tabular}{ccccccccc}
			\hline
			&	\multicolumn{3}{c}{Direct L1 scheme for original model}&&\multicolumn{3}{c}{Scheme \eqref{mm15}--\eqref{mm8} with $n=3$}\\
			\cline{2-4}	\cline{6-8}
			&	$ M $ & ${\rm Error}_M(\tau, h)$   & ${\rm Rate}_{\tau}$ 	&&	$ M $ & ${\rm Error}_M(\tau, h)$   & ${\rm Rate}_{\tau}$   \\
			\hline
			&	64     & $ 5.23 \times 10^{-4} $ &  *  	    &&	64      & $1.06 \times 10^{-5} $ &  *  \\
			&	128     & $ 1.71 \times 10^{-4} $ &  1.61  	    &&	128      & $3.34 \times 10^{-6} $ &  1.67  \\
			&	256     & $ 6.21 \times 10^{-5} $ & 1.46       &&	256      & $1.04 \times 10^{-6} $ & 1.68   \\
			$\alpha=0.25$	&	512    & $ 4.75 \times 10^{-5} $  & 0.39      &&	512    & $3.21 \times 10^{-7} $  & 1.69 \\
			&	1024    & $ 2.92 \times 10^{-5} $ & 0.70     &&	1024     & $9.85 \times 10^{-8} $ & 1.70\\
			&Theory&&1.75&&&&1.75\\
			\hline
			&	512     & $1.88 \times 10^{-4} $ &  *       &&	512     & $7.11\times 10^{-5} $ &  *   \\
			&	1024      & $8.12 \times 10^{-5} $ & 1.21     &&	1024     & $3.19\times 10^{-5} $ & 1.16 \\
			$\alpha=0.75$	&		2048    & $3.49 \times 10^{-5} $  & 1.22      &&	2048    & $1.43\times 10^{-5} $ & 1.16\\
			&   4096   & $1.49 \times 10^{-5} $ & 1.22    &&	4096    & $ 6.33 \times 10^{-6} $ & 1.17\\
			&	8192   & $6.37 \times 10^{-6} $ & 1.23     &&	8192    & $ 2.79 \times 10^{-6} $ & 1.18\\
		
			&Theory&&1.25&&&&1.25\\
			\hline
		\end{tabular}
		\label{tab:mm2}%
	\end{table}

	\subsection{Integrodifferential equation}\label{sec42}
	In this section, we consider the following integrodifferential equation, which has been widely used in modelling, e.g. heat transfer in materials with memory \cite{Pruss,Renardy1,Yi} 
	\begin{equation}\label{pm1}
		\begin{array}{ll}
			&~~~\p_t u(x,t)-I_t^\alpha\Delta u(x,t)=f(x,t),~~(x,t)\in \Omega\times (0,T];\\[0.05in]
			&u(x,t)=0,~~(x,t)\in \p\Omega\times [0,T];~~u(x,0)=u_0(x),~~x\in\Omega.
		\end{array}
	\end{equation}
	Compared with the fractional relaxation equation, Volterra integral equation or subdiffusion, which are models of order between $0$ and $1$ in time, the integrodifferential equation is a model of order between $1$ and $2$ such that the properties would be significantly different. Thus we give a detailed mathematical and numerical analysis.
	We will see that the MSD for model (\ref{pm1}) with $n=1$, that is, we only separate one singular term from $u$, will significantly improve the regularity and then facilitate the numerical approximation, which indicates that the MSD is particularly suitable for integrodifferential problems.

\subsubsection{Mathematical analysis}

We apply the variable substitution $w= u-u_0$ to get
	\begin{equation}\label{pm3}
		\begin{array}{cc}
			\p_t w(x,t)-I_t^\alpha\Delta w(x,t)=f(x,t)+\Delta u_0(x)\beta_{\alpha+1}(t),~~(x,t)\in \Omega\times (0,T];\\[0.05in]
			w(x,t)=0,~~(x,t)\in \p\Omega\times [0,T];~~w(x,0)=0,~~x\in\Omega.
		\end{array}
	\end{equation}
Now we intend to apply the MSD for model (\ref{pm3}). In general, one could follow the spirit of (\ref{mh1}) or (\ref{sub2}) to separate $n$ singular terms from $w$. However, a larger $n$ may inquire a higher spatial smoothness on the data. For integrodifferential equation, however, it usually suffices to select $n=1$ due to the substantial improvement of the solution regularity. Specifically, let
	\begin{align}\label{pm31}
		w=v+I_t^1(f+\Delta u_0\beta_{\alpha+1}),
	\end{align}
and we assume $f,\Delta u_0\in \dot H^{1/2+\varepsilon}$ for $t\in [0,T]$ for some $0<\varepsilon\ll 1$ throughout this section such that $f=\Delta u_0=0$ for $(x,t)\in \p\Omega\times [0,T]$.
	We then invoke (\ref{pm31}) in (\ref{pm3}) to get
		\begin{align}
			&\p_t v(x,t)-I_t^\alpha\Delta v(x,t)=F(x,t):=\Delta I_t^{1+\alpha}(f+\Delta u_0\beta_{\alpha+1})(x,t),~~(x,t)\in \Omega\times (0,T];\nonumber\\[0.05in]
		&\qquad\qquad\qquad	v(x,t)=0,~~(x,t)\in \p\Omega\times [0,T];~~v(x,0)=0,~~x\in\Omega.\label{pm2}
		\end{align}

The following theorem provides estimates for time derivatives of $v$, which, compared with the classical result $\p_t^m u\sim t^{\alpha-(m-1)}$ for $m=2,3$ (see e.g. \cite[Page 64]{Mc}), becomes much smoother in time.
	\begin{theorem}\label{thm4} The following estimates hold:
	\begin{itemize}
	\item	Suppose  $\beta_{(1+\alpha)\varepsilon}*\|\Delta f\|_{ H^{2\varepsilon}}^2<\infty$ for $t\in [0,T]$ for some $0<\varepsilon\ll 1$, then
		$
			\|\p_tv(\cdot,t)\|_{\dot H^2}\leq Q$ for $t\in [0,T]$;	
			\item Suppose $\|\Delta f(\cdot,0)\|_{\dot H^{\frac{2\alpha}{1+\alpha}}}<\infty$, $\beta_{(1+\alpha)\varepsilon}*\|\p_t\Delta f\|_{ H^{2\varepsilon}}^2<\infty$ for $t\in [0,T]$ and some $0<\varepsilon\ll 1$ and $\|\Delta^2 u_0\|_{\dot H^{\frac{2(1-\alpha)}{1+\alpha}}}<\infty$,  then $
			\|\p_t^2v(\cdot,t)\|_{\dot H^2}\leq Q$ for $t\in (0,T]$;
\item Suppose  $\|\Delta f(\cdot,0)\|_{\dot H^{2\varepsilon}}<\infty$, $\|\Delta^2 u_0\|_{\dot H^{2\varepsilon}}<\infty$, and
		$I_t^{\alpha}\|\p_t\Delta f(\cdot,t)\|<\infty$ and $\beta_{(\alpha+1)\varepsilon}*\big\| I_t^{\alpha}\p_t\Delta f(\cdot,t)\big\|_{\dot H^{2\varepsilon}}<\infty$ for $t\in [0,T]$ and some $0<\varepsilon\ll 1$, then
	$\|\p_t^3v(\cdot,t)\|\leq Qt^{\alpha-1}$ for $t\in (0,T]$.
			\end{itemize}
	\end{theorem}
	\begin{proof} We apply the solution operator $E(t)$ and its time derivatives, whose properties are characterized by Mittag-Leffler functions.
		The solution of (\ref{pm2}) could be expressed via the solution operator $E(t)$ \cite{Mc}, i.e., $v=\int_0^t  E(t-s) I_s^{1+\alpha}\Delta F(x,s)ds$ with $E(t)q:=E_{1+\alpha,1}(-\lambda_jt^{1+\alpha})(q,\phi_j)\phi_j$. Differentiate this equation in time and apply the boundedness of $\Delta f$ to get
		\begin{align*}
			\p_tv&=\int_0^t  E(t-s) I_s^{\alpha}\Delta F(x,s)ds,~~
			\p_t^2v=\int_0^t  E(t-s) \p_s(I_s^{\alpha}\Delta F(x,s))ds\\
			&=\int_0^t  E(t-s)\big[\beta_{\alpha}\Delta f(x,0)+ I_s^{\alpha}\p_s\Delta (f(x,s)+\beta_{\alpha+1}(s)\Delta u_0)\big]ds,\\
			\p_t^3v&=\beta_{\alpha}\Delta f(x,0)+ I_t^{\alpha}\p_t\Delta (f(x,t)+\beta_{\alpha+1}(t)\Delta u_0)\\
			& \quad +\int_0^t  E'(t-s)\big[\beta_{\alpha}\Delta f(x,0)+ I_s^{\alpha}\p_s\Delta (f(x,s)+\beta_{\alpha+1}(s)\Delta u_0)\big]ds.
		\end{align*}
		To bound $\|\p_t v\|_{\dot H^2}$ under weak regularity of the data, we evaluate the right-hand side of $\p_t v$ as
		\begin{align}
			\p_t v&=\int_0^t \sum_{j=1}^\infty E_{1+\alpha,1}(-\lambda_jt^{1+\alpha})(I_s^{\alpha}\Delta F(x,s),\phi_j)\phi_jds \nonumber\\
			&=\sum_{j=1}^\infty E_{1+\alpha,1}(-\lambda_jt^{1+\alpha})*\beta_{\alpha}*(\Delta F(x,s),\phi_j)\phi_j \nonumber\\
			&=\sum_{j=1}^\infty t^\alpha E_{1+\alpha,1+\alpha}(-\lambda_jt^{1+\alpha})*(\Delta f+\Delta^2 u_0\beta_{\alpha+1},\phi_j)\phi_j \nonumber\\
			&=\sum_{j=1}^\infty t^\alpha E_{1+\alpha,1+\alpha}(-\lambda_jt^{1+\alpha})*(\Delta f,\phi_j)\phi_j \nonumber\\
			&\qquad+\sum_{j=1}^\infty t^{2\alpha+1} E_{1+\alpha,2(1+\alpha)}(-\lambda_jt^{1+\alpha})(\Delta^2 u_0,\phi_j)\phi_j=:A_1+A_2. \nonumber
		\end{align}
	By $\alpha-(1+\alpha)(1-\varepsilon)>-1$ for any $0<\varepsilon\ll 1$, we bound $A_1$ as
		\begin{align*}
			\|A_1\|_{\dot H^2}^2&=\sum_{j=1}^\infty\lambda_j^2\big[ t^\alpha E_{1+\alpha,1+\alpha}(-\lambda_jt^{1+\alpha})*(\Delta f(x,s),\phi_j)\big]^2\\
			&=\sum_{j=1}^\infty\big[ \lambda_j^\varepsilon (\lambda_jt^{1+\alpha})^{1-\varepsilon} t^{\alpha-(1+\alpha)(1-\varepsilon)} E_{1+\alpha,1+\alpha}(-\lambda_jt^{1+\alpha})*(\Delta f(x,s),\phi_j)\big]^2\nonumber\\
			&\leq Q\sum_{j=1}^\infty\big[ \lambda_j^\varepsilon  t^{\alpha-(1+\alpha)(1-\varepsilon)}*|(\Delta f(x,s),\phi_j)|\big]^2\\
			&\leq Q\sum_{j=1}^\infty\lambda_j^{2\varepsilon}t^{\alpha-(1+\alpha)(1-\varepsilon)}*(\Delta f(x,s),\phi_j)^2=Qt^{\alpha-(1+\alpha)(1-\varepsilon)}*\|\Delta f\|_{\dot H^{2\varepsilon}}^2\\
			&\leq Q\beta_{(1+\alpha)\varepsilon}*\|\Delta f\|_{ H^{2\varepsilon}}^2.
		\end{align*}
		Then we apply
		\begin{align}
			&|\lambda_jt^{2\alpha+1} E_{1+\alpha,2(1+\alpha)}(-\lambda_jt^{1+\alpha})| \nonumber\\
			&\qquad=|\lambda_j^{\frac{-\alpha}{1+\alpha}}(\lambda_jt^{1+\alpha})^{\frac{1+2\alpha}{1+\alpha}}E_{1+\alpha,2(1+\alpha)}(-\lambda_jt^{1+\alpha})|\leq Q\lambda_j^{\frac{-\alpha}{1+\alpha}} \nonumber
		\end{align}
		to bound $A_2$ as
		\begin{align}
			\|A_2\|_{\dot H^2}^2\leq Q\sum_{j=1}^\infty\lambda_j^{\frac{-2\alpha}{1+\alpha}}(\Delta^2 u_0,\phi_j)^2=Q\|\Delta u_0\|_{\dot H^{\frac{-2\alpha}{1+\alpha}}}^2. \nonumber
		\end{align}
		
		Similar to the estimates above, we first simplify the expression of $\p_t^2v$ as
		\begin{align}
			\p_t^2 v&=\sum_{j=1}^\infty t^{\alpha}E_{1+\alpha,1+\alpha}(-\lambda_jt^{1+\alpha})(\Delta f(x,0),\phi_j)\phi_j \nonumber\\
			&\quad+\sum_{j=1}^\infty t^{\alpha}E_{1+\alpha,1+\alpha}(-\lambda_jt^{1+\alpha})*(\p_t\Delta f(x,t),\phi_j)\phi_j \nonumber\\
			&\quad+\sum_{j=1}^\infty t^{2\alpha}E_{1+\alpha,1+2\alpha}(-\lambda_jt^{1+\alpha})(\Delta^2 u_0,\phi_j)\phi_j, \nonumber
		\end{align}
		and then bound $\|\p_t^2v\|_{\dot H^2}$ as
		\begin{align}
			\|\p_t^2v\|_{\dot H^2}^2&\leq Q\sum_{j=1}^\infty \lambda_j^{\frac{2\alpha}{1+\alpha}}(\Delta f(x,0),\phi_j)^2+Q\beta_{(1+\alpha)\varepsilon}*\|\p_t\Delta f\|_{ H^{2\varepsilon}}^2 \nonumber\\
			&\quad+\sum_{j=1}^\infty\lambda_j^{\frac{2(1-\alpha)}{1+\alpha}} (\Delta^2 u_0,\phi_j)^2 \nonumber\\
			&=Q\|\Delta f(\cdot,0)\|_{\dot H^{\frac{2\alpha}{1+\alpha}}}^2+Q\beta_{(1+\alpha)\varepsilon}*\|\p_t\Delta f\|_{ H^{2\varepsilon}}^2+Q\|\Delta^2 u_0\|_{\dot H^{\frac{2(1-\alpha)}{1+\alpha}}}^2. \nonumber
		\end{align}
		Finally, we apply $\|E'q\|\leq Qt^{(1+\alpha)\varepsilon-1}\|q\|_{\dot H^{2\varepsilon}}$ \cite[Theorem 5.5]{Mc} to get
		\begin{align*}
			\|\p_t^3v\|&\leq \beta_{\alpha}\|\Delta f(\cdot,0)\|+Q I_t^{\alpha}\|\p_t\Delta (f(\cdot,t)+\beta_{\alpha+1}(t)\Delta u_0)\|\\
			&~~+Q\beta_{(\alpha+1)\varepsilon}*\big\|\beta_{\alpha}\Delta f(\cdot,0)+ I_s^{\alpha}\p_t\Delta (f(\cdot,t)+\beta_{\alpha+1}(t)\Delta u_0)\big\|_{\dot H^{2\varepsilon}}\\
			&\leq  Qt^{\alpha-1}(\|\Delta f(\cdot,0)\|_{\dot H^{2\varepsilon}}+\|\Delta^2 u_0\|_{\dot H^{2\varepsilon}})\\
			&~~+ QI_t^{\alpha}\|\p_t\Delta f(\cdot,t)\|+Q\beta_{(\alpha+1)\varepsilon}*\big\| I_t^{\alpha}\p_t\Delta f(\cdot,t)\big\|_{\dot H^{2\varepsilon}},
		\end{align*}
		which completes the proof.
	\end{proof}
	
	\subsubsection{Semi-discrete scheme}
	We introduce the following trapezoid convolution quadrature rule \cite{Cuesta,Qiu}
	\begin{align}\label{tcq}
		\mathcal{Q}_m^{(\alpha)} (\varphi) = \tau^{\alpha} \sum_{p=0}^{m} \omega_{p}^{(\alpha)} \varphi^{m-p} + \chi_{m}^{(\alpha)} \varphi^0,
	\end{align}
	where the weights $\omega_{p}^{(\alpha)}$ are determined from the generating function
	\begin{align*}
		(\delta(z))^{-\alpha}  = \sum_{m=0}^{\infty}  \omega_{m}^{(\alpha)} z^{m}, \quad  \delta(z) = \frac{2(1-z)}{1+z}.
	\end{align*}
	In addition, the starting weights are determined by taking $\varphi=1$
	\begin{align*}
		\chi_{m}^{(\alpha)} = \frac{(t_m)^{\alpha}}{\Gamma(1+\alpha)} - \sum_{p=0}^{m} \omega_{p}^{(\alpha)}, \quad 1\leq m \leq M.
	\end{align*}
Then we denote the quadrature error $(R_2)^m  := I_t^\alpha \Delta v(x,t_m) - \mathcal{Q}_m^{(\alpha)} (\Delta v(x,\cdot))$. To discretize $\p_t v(x,t_m)$, we denote $v^m=v(x,t_m)$ and
	\begin{align*}
		\delta_t v^m = \frac{v^m-v^{m-1}}{\tau}, \quad v^{m-\frac{1}{2}} = \frac{v^m+v^{m-1}}{2},  \quad m\geq 1.
	\end{align*}
	We then apply the Crank-Nicolson method to approximate $\p_t v(x,t_m)$ as
	\begin{align}\label{err02}
		\frac{1}{2}[ \p_t v(x,t_m) + \p_t v(x,t_{m-1}) ] = \delta_t v^m + (R_3)^{m-1/2}.
	\end{align}
We invoke \eqref{tcq} and \eqref{err02} to discretize \eqref{pm2} as
	\begin{align}
		\delta_t v^m &  - \frac{1}{2}\left[ \mathcal{Q}_m^{(\alpha)} (\Delta v(x,\cdot)) + \mathcal{Q}_{m-1}^{(\alpha)} (\Delta v(x,\cdot))  \right] \nonumber \\
		& = \frac{1}{2}[F(x,t_m)+F(x,t_{m-1})] -  [(R_2)^{m-1/2} +  (R_3)^{m-1/2}], \label{mm02}
	\end{align}
	with $v^0=v(x,0)=0$. We drop the truncation errors $(R_2)^{m-1/2} +  (R_3)^{m-1/2}$ to get the semi-discrete scheme
	\begin{align}
		& \delta_t V^m   - \tau^{\alpha} \sum_{p=0}^{m} \omega_{p}^{(\alpha)} \Delta V^{m-p-\frac{1}{2}}  = \frac{1}{2}[F(x,t_m)+F(x,t_{m-1})], \quad V^0=V^{-1}=0,  \label{mm03} \\
       &  \qquad\qquad\qquad\qquad  U^m = V^m +I_t^1(f+\Delta u_0\beta_{\alpha+1})(x,t_m) + u_0(x). \nonumber
	\end{align}

	
	\begin{theorem}\label{thmqq}
		Under the assumptions of Theorem \ref{thm4}, the following error estimate holds
		\begin{equation*}
			\begin{array}{l}
				\|u^m -U^m\|\leq    Q \tau^2, \quad 1\leq m \leq M.
			\end{array}
		\end{equation*}
	\end{theorem}
	\begin{proof} Let $\rho^m=v^m-V^m$. Since $u^m -U^m=	v^m -V^m$, it suffices to estimate $\rho$. We subtract \eqref{mm03} from \eqref{mm02} to get the following error equation
	\begin{align}
		\delta_t \rho^m &  - \tau^{\alpha} \sum_{p=0}^{m} \omega_{p}^{(\alpha)} \Delta \rho^{m-p-\frac{1}{2}}  = (R_2)^{m-1/2} +  (R_3)^{m-1/2}, \quad \rho^0=0.  \label{mm06}
	\end{align}
 We then take the inner product of \eqref{mm06} with $2\tau v^{m-\frac{1}{2}}$ and sum the resulting equation from $m=1$ to $m=m_0$ to get
		\begin{align*}
			\|\rho^{m_0}\|^2 - \|\rho^{0}\|^2 & + 2\tau^{1+\alpha}\sum_{m=1}^{m_0} \sum_{p=0}^{m} \omega_{p}^{(\alpha)} ( \nabla \rho^{m-p-\frac{1}{2}}, \nabla \rho^{m-\frac{1}{2}}) \\
			& \leq 2\tau \sum_{m=1}^{m_0} \|(R_2)^{m-1/2} +  (R_3)^{m-1/2}\| \|v^{m-\frac{1}{2}}\|.
		\end{align*}
	By \cite[Lemma 5]{Qiu} and $\rho^0=0$, we have
		\begin{align*}
			& \sum_{m=1}^{m_0} \sum_{p=0}^{m} \omega_{p}^{(\alpha)} ( \nabla \rho^{m-p-\frac{1}{2}}, \nabla \rho^{m-\frac{1}{2}}) = \sum_{m=0}^{m_0} \sum_{p=0}^{m} \omega_{p}^{(\alpha)} ( \nabla \rho^{m-p-\frac{1}{2}}, \nabla \rho^{m-\frac{1}{2}}) \geq 0.
		\end{align*}
We apply this and choose a suitable $m^*$ satisfying $\ds\|\rho^{m^*}\|=\max_{1\leq m \leq m_0}\|\rho^m\|$ to obtain
		\begin{align*}
			\|\rho^{m_0}\| \leq \|\rho^{m^*}\| \leq 2\tau \sum_{m=1}^{m^*} \|(R_2)^{m-\frac{1}{2}} +  (R_3)^{m-\frac{1}{2}}\| \leq 2\tau \sum_{m=1}^{m_0} (\|(R_2)^{m-\frac{1}{2}}\| + \| (R_3)^{m-\frac{1}{2}}\|).
		\end{align*}
By \cite[Lemma 4]{Qiu}, the quadrature error $(R_2)^m$  could be bounded as
	\begin{align*}
		\|(R_2)^m\| \leq Q & \Big[ \tau^2 t_m^{\alpha-1}\|\Delta \p_t v(\cdot,0)\|+\tau^{\alpha+1} \int_{t_{m-1}}^{t_m}\| \Delta \p_t^2 v(\cdot,s)\|ds \\
		& \qquad\qquad   + \tau^2\int_{0}^{t_{m-1}}(t_m-s)^{\alpha-1} \|\Delta \p_t^2 v(\cdot,s)\|ds  \Big].
	\end{align*}
Then we apply the triangle inequality to bound $(R_3)^{m-\frac{1}{2}}$
	\begin{align*}
		\|(R_3)^{m-\frac{1}{2}}\| \leq \left\| \frac{1}{2}[ \p_t v(\cdot,t_m) + \p_t v(\cdot,t_{m-1}) ] - \p_t v(\cdot,t_{m-\frac{1}{2}}) \right\| +  \left\| \p_t v(\cdot,t_{m-\frac{1}{2}})-\delta_t v^m \right\|.
	\end{align*}
By Taylor's expansion we have
	\begin{align}\label{err03}
		\|(R_3)^{\frac{1}{2}}\| \leq \int_{0}^{\tau} \|\p_t^2v(\cdot,s)\| ds, \quad \|(R_3)^{m-\frac{1}{2}}\| \leq  \tau \int_{t_{m-1}}^{t_m} \|\p_t^3 v(\cdot,s)\| ds, \quad m\geq 2.
	\end{align}
Based on the above estimates, we invoke Theorem \ref{thm4} to obtain
		\begin{align}
			& \tau \sum_{m=1}^{m_0} \|(R_2)^{m-\frac{1}{2}}\|   \leq Q\tau^2 \left[ \tau \sum_{m=1}^{m_0}t_m^{\alpha-1} + \tau^{\alpha} + 1 \right]  \leq Q \tau^2 \int_{0}^{t_{m_0}}  s^{\alpha-1} ds \leq Q \tau^2,  \label{mm07} \\
			& \tau \sum_{m=1}^{m_0} \|(R_3)^{m-\frac{1}{2}} \| \leq  \tau\int_{0}^{\tau} s^{\alpha} ds + \tau^2 \int_{\tau}^{t_{m_0}} s^{\alpha-1} ds \leq Q \tau^{2}, \label{mm08}
		\end{align}
which completes the proof.
	\end{proof}

	\begin{remark}
	It is worth mentioning that if we apply the preceding discretizations for the original model \eqref{pm1}, then only $O(\tau^{1+\alpha})$ accuracy could be obtained. The reason is that when we analyze the truncation errors in (\ref{err03}), the regularity of the solutions to \eqref{pm1}, that is,
$
			\|\p_t^2u(\cdot,t)\|_{\dot H^2} \leq Qt^{\alpha-1}$ and $   \|\p_t^3u(\cdot,t)\|\leq Qt^{\alpha-2}$ \cite{Mc},
	leads to
			\begin{align*}
			\tau \sum_{m=1}^{M} \|(R_3)^{m-\frac{1}{2}} \| \leq  Q\tau\int_{0}^{\tau} s^{\alpha-1} ds + Q\tau^2 \int_{\tau}^{t_M} s^{\alpha-2} ds \leq Q \tau^{\alpha+1}.
		\end{align*}
	This demonstrates the advantages of the MSD in improving the numerical accuracy.
	\end{remark}

\subsubsection{Fully-discrete scheme}
For spatial discretization, we integrate \eqref{mm02} multiplied by $\hat \chi \in H_0^1$ on $\Omega$ to get the weak formulation for $1\leq m \leq M$
	\begin{align}
		&\left(  \delta_t v^m, \hat{\chi} \right) + \tau^{\alpha} \sum_{p=0}^{m} \omega_{p}^{(\alpha)}  \big( \nabla v^{m-p-\frac{1}{2}} , \nabla\hat{\chi} \big)  =  \big( \frac{1}{2}[F(x,t_m)+F(x,t_{m-1})], \hat{\chi} \big)\nonumber\\
		&\qquad - \big((R_2)^{m-\frac{1}{2}}+(R_3)^{m-\frac{1}{2}}, \hat{\chi} \big). \label{FEM:e1}
	\end{align}
We then omit the truncation errors to establish the fully-discrete finite element scheme: find $V^m_h \in S_h$ with $V^0_h : =0$ such that for $\hat{\chi} \in S_h$ and $1\leq m \leq M$
	\begin{align}
		\left(  \delta_t V_h^m, \hat{\chi} \right) + \tau^{\alpha} \sum_{p=0}^{m} \omega_{p}^{(\alpha)}  \big( \nabla V_h^{m-p-\frac{1}{2}} , \nabla\hat{\chi} \big)  =  \big( \frac{1}{2}[F(x,t_m)+F(x,t_{m-1})], \hat{\chi} \big).   \label{lkx01}
	\end{align}
	Based on $w=u-u_0$ and \eqref{pm31}, the numerical solution for \eqref{pm1} can be recovered as
	\begin{align}\label{lmm1}
			U_h^m=V_h^m + I^1_t( f+\Delta u_0\beta_{\alpha+1})(x,t_m)+u_0.
	\end{align}

To perform error estimate for the fully-discrete scheme, we define the Ritz projection $I_h: H^1_0\rightarrow S_h$ by  $
     (\nabla (I_h v - v), \nabla \hat{\chi} ) = 0 $ for $\hat{\chi} \in S_h $ such that $\|v-I_hv\|\leq Qh^2\|v\|_{H^2}$.
	
	\begin{theorem}
	Under the assumptions of Theorem \ref{thm4}, we have
		\begin{equation*}
			\begin{array}{l}
				\|u^m -U_h^m\| \leq    Q (\tau^2+h^2),   \quad 1\leq m \leq M.
			\end{array}
		\end{equation*}
	\end{theorem}
\begin{proof}
As  $u^m -U_h^m=v^m -V_h^m$, it suffices to estimate $v^m -V_h^m$. Let $v^m- V_h^m = \xi^m - \eta^m$ with $\xi^m = I_h v^m - V_h^m \in S_h$ and $\eta^m = I_h v^m - v^m$.
	     We subtract \eqref{lkx01} from \eqref{FEM:e1} and set $\hat \chi = \xi^n$ to get
	\begin{align*}
		\left(  \delta_t \xi^m, \xi^n \right) + \tau^{\alpha} \sum_{p=0}^{m} \omega_{p}^{(\alpha)}  \left( \nabla \xi^{m-p-\frac{1}{2}} , \nabla \xi^n \right)  =  \left( \delta_t \eta^m - (R_2)^{m-\frac{1}{2}}-(R_3)^{m-\frac{1}{2}}, \xi^n \right).
	\end{align*}
We first follow \cite{Tho} to obtain $\|\p_t^{\ell} \eta(t) \| \leq Q h^2 \|\p_t^{\ell} v\|_{H^2}$ with $\ell=0,1$. Then the proof could be performed in analogous as that of Theorem \ref{thmqq}, that is, we use \eqref{mm07}, \eqref{mm08},  and
	\begin{align*}
		\tau \sum_{m=1}^{m_0} \|\delta_t \eta^m\| \leq Q \tau\sum_{m=1}^{m_0} \frac{1}{\tau} \int_{t_{m-1}}^{t_m} \|\p_t \eta(t) \|dt  \leq
		Q \int_{0}^{t_{m_0}} h^2 \|\p_t v\|_{H^2}dt \leq Qh^2
	\end{align*}
to get  $\|\xi^m\|\leq Q(\tau^2+h^2)$, which, together with $\|\eta^m\|\leq Qh^2 \|v\|_{H^2}\leq Qh^2$, completes the proof.
\end{proof}

	\subsubsection{Numerical experiment}
 Let $T=1$, $\Omega=(0,1)$, $u_0=\sin(\pi x)$ and $f=t^\alpha \sin(\pi x)$, and we use the uniform partition in space with $h=1/J$ for some $1<J\in\mathbb N$. We measure errors as (\ref{Rate}) and and convergence orders by the two mesh principle as (\ref{zz2}). We compare the trapezoidal convolution quadrature method (TCQM) for the original model \eqref{pm1} (see e.g. \cite{Cuesta,Qiu,Xu}) and the MSD-based TCQM, i.e., the scheme \eqref{lkx01}--\eqref{lmm1} in
 Table \ref{tab:example2} under $J=32$, which indicates that the MSD improves the numerical accuracy from $O(\tau^{1+\alpha})$ to $O(\tau^2)$ due to the separation of singularity, which again demonstrates its advantages.
	
	\begin{table}
		\centering \footnotesize
		\caption{Numerical methods for integrodifferential equation.}
		\renewcommand{\arraystretch}{1.1} 
		\begin{tabular}{ccccccccc}
			\hline
			&	\multicolumn{3}{c}{TCQM for original model}&&\multicolumn{3}{c}{Scheme \eqref{lkx01}--\eqref{lmm1}}\\
			\cline{2-4}	\cline{6-8}
			&	$ M $ & ${\rm Error}_M(\tau, h)$   & ${\rm Rate}_{\tau}$ 	&&	$ M $ & ${\rm Error}_M(\tau, h)$   & ${\rm Rate}_{\tau}$  \\
			\hline	
			&	128    & $ 2.40 \times 10^{-3} $  & *     &&	128    & $4.07 \times 10^{-5} $  &* \\
			&	256    & $ 1.01 \times 10^{-3} $ & 1.25      &&	256    & $1.05 \times 10^{-5} $ & 1.96\\
			$\alpha=0.25$			&	512    & $ 4.23 \times 10^{-4} $ & 1.25      &&	512    & $2.67 \times 10^{-6} $ & 1.97\\
			&	1024    & $ 1.78 \times 10^{-4} $ & 1.25      &&	1024    & $6.76 \times 10^{-7} $ & 1.98 \\
			&	2048     & $ 7.48 \times 10^{-5} $ &  1.25  	    &&	2048     & $1.71 \times 10^{-7} $ &  1.99  \\
			&Theory&&1.25&&&&2.00\\
			\hline
			&	128     & $ 1.47 \times 10^{-4} $ & *      &&	128    & $7.60 \times 10^{-5} $ & *\\
			&	256    & $ 4.38 \times 10^{-5} $ & 1.74      &&	256    & $ 1.90 \times 10^{-5} $ & 2.00\\
			$\alpha=0.75$		&	512    & $ 1.31 \times 10^{-5} $ & 1.74      &&	512    & $ 4.75 \times 10^{-6} $ & 2.00\\
			&	1024    & $ 3.95 \times 10^{-6} $ & 1.73      &&	1024    & $1.19 \times 10^{-6} $ & 2.00\\
			&	2048     & $ 1.19 \times 10^{-6} $ &  1.73  	    &&	2048     & $2.97 \times 10^{-7} $ &  2.00  \\
			&Theory&&1.75&&&&2.00\\
			\hline
		\end{tabular}
		\label{tab:example2}%
	\end{table}

	\subsection{Diffusion-wave equation}
	We consider the following diffusion wave equation with $1<\gamma<2$, which models the propagation of mechanical diffusive waves in viscoelastic media which exhibit a power-law creep \cite{Lin,Mai1,Mai2,ZhangZ}
\begin{gather}
	\hspace{-0.2in}\left( \beta_{2-\gamma} * \partial_t^2 u \right)(x,t) - \Delta u(x,t) = f(x,t), \quad (x,t) \in \Omega \times (0,T], \label{q1} \\
	u(x,t) = 0,~ (x,t) \in \partial\Omega \times [0,T];~
	u(x,0) = u_0(x), ~ \partial_t u(x,0) = \bar{u}_0(x), ~ x \in \Omega. \label{q3}
\end{gather}
For \eqref{q1}--\eqref{q3}, we apply the variable substitution $w=u-u_0$ to get
\begin{gather*}
	\left(\beta_{2-\gamma} * \p_t^{2} w \right) (x,t)  - \Delta w(x,t) = f(x,t) + \Delta u_0(x),~~(x,t)\in \Omega\times (0,T], \\
	\hspace{-0.3in}w(x,t)=0, ~  (x,t)\in \p\Omega\times [0,T];~
	w(x,0)=0,~  \p_t w(x,0)= \bar{u}_0(x),  ~ x\in\Omega.
\end{gather*}
We then calculate its convolution with $\beta_{\gamma-1}$ to obtain
\begin{gather*}
	\p_t w (x,t)  - I_t^{\gamma-1}\Delta w(x,t) =  I_t^{\gamma-1}f(x,t) + \beta_{\gamma}(t) \Delta u_0(x)  + \bar{u}_0(x),~~(x,t)\in \Omega\times (0,T],  \\
	w(x,t)=0,~  (x,t)\in \p\Omega\times [0,T];~
	w(x,0)=0, ~ \p_t w(x,0)= \bar{u}_0(x),  ~ x\in\Omega,
\end{gather*}
which is an integrodifferential equation considered in Section \ref{sec42}. Thus, the MSD could facilitate numerical computation by improving the solution regularity. For instance, if
we separate two singular terms from $w$, i.e. we take
$$
			w = v + \sum_{i=0}^{1}(\Delta I_t^{\gamma})^iI_t^1 \big( I_t^{\gamma-1}f + \beta_{\gamma} \Delta u_0  + \bar{u}_0 \big),
		$$ 
		a similar proof as that of Theorem \ref{thm4} could give the boundedness of $\p_t^2v$ and $\p_t^3v$ under suitable assumptions on the data, which is consistent with the regularity assumptions in the classical work \cite{Sun} on the diffusion-wave equation and thus supports the $O(\tau^{2-\alpha})$ accuracy of the scheme that is proposed in \cite{Sun} and widely adopted in the literature.

\section{Concluding remarks}\label{Sec6}

This work proposes an MSD method for nonlocal-in-time problems, including the fractional relaxation equation,Volterra integral equation, subdiffusion, integrodifferential equation and diffusion-wave equation, which separates singularities from the solutions by simple variable substitutions. Several concrete examples are given to demonstrate the advantages of MSD in numerical computation.

Apart from the considered examples, the MSD could also be applied for more general nonlocal-in-time problems or other numerical methods. For instance, one could consider models with variable coefficients (e.g. the variable diffusivity tensor in subdiffusion) or parameters (e.g. the variable exponent in fractional operators \cite{Zheng}). From the aspect of numerical method, one could improve the numerical accuracy of the L2-$1_\sigma$ scheme \cite{Liao1} on the uniform mesh for subdiffusion, or relax the singularity of the mesh for the averaged PI rule \cite{Mus} or eliminate the logarithmic factor in error estimate in \cite{Mus} for integrodifferential equation. We will investigate these interesting topics in the near future.

\section*{Acknowledgments}
This work was partially supported by the Natural Science Foundation of Shandong Province (No. ZR2025QB01),  the National Natural Science Foundation of China (No. 12301555), the National Key R\&D Program of China (No. 2023YFA1008903),  the Taishan Scholars Program of Shandong Province (No. tsqn202306083), the Postdoctoral Fellowship Program of CPSF (No. GZC20240938).

\end{document}